\documentclass[11pt, reqno]{amsart}
\usepackage{amssymb, amsthm, amsmath, amsfonts,mathtools,mathrsfs,calc,mathabx}
\usepackage{array, epsfig}
\usepackage{bbm}
\usepackage{enumitem}

\usepackage{hyperref}
\usepackage[numbers,square]{natbib}
\usepackage{color}
\usepackage{ stmaryrd }
\allowdisplaybreaks

\numberwithin{equation}{section}

  \evensidemargin
\oddsidemargin
\usepackage{upgreek}

\newcommand*\xbar[1]{%
   \hbox{%
     \vbox{%
       \hrule height 0.5pt 
       \kern0.25ex
       \hbox{%
         \kern-0.05em
         \ensuremath{#1}%
         \kern-0.1em
       }%
     }%
   }%
}

\newcommand{\ind}{\mathbbm{1}}

\newcommand{\var}{\operatorname{var}}

\newcommand{\eps}{\varepsilon}

\newcommand{\stirling}[2]{\genfrac{[}{]}{0pt}{}{#1}{#2}}

\def\CC{\mathbb{C}}

\def\EE{\mathbb{E}}

\def\NN{\mathbb{N}}

\def\bP{\mathbb{P}}

\def\RR{\mathbb{R}}
\def\RRd1{\mathbb{R}^{d+1}}

\def\ZZ{\mathbb{Z}}

\def\bE{\mathbf{E}}

\def\bP{\mathbf{P}}

\def\cF{\mathcal{F}}

\def\cI{\mathcal{I}}

\newcommand{\todistr}{\overset{d}{\underset{n\to\infty}\longrightarrow}}

\newcommand{\toprobab}{\overset{P}{\underset{n\to\infty}\longrightarrow}}

\newcommand{\toas}{\overset{a.s.}{\underset{n\to\infty}\longrightarrow}}

\newcommand{\ton}{\overset{}{\underset{n\to\infty}\longrightarrow}}

\newcommand{\tosim}{\underset{n\to\infty}{\sim}}

\def\dint{\textup{d}}

\def\pos{\textup{pos}}

\def\Bern{\textup{Bern}}
\newcommand{\relint}{\textup{relint}}

\theoremstyle{plain}
\newtheorem{theorem}{Theorem}[section]

\theoremstyle{definition}

\theoremstyle{remark}
\newtheorem{remark}[theorem]{Remark}

\begin{document}

\author{Thomas Godland}
\address{Thomas Godland: Institut f\"ur Mathematische Stochastik,
Westf\"alische Wilhelms-Universit\"at M\"unster,
Orl\'eans-Ring 10,
48149 M\"unster, Germany}
\email{thomas.godland@uni-muenster.de}

\author{Zakhar Kabluchko}
\address{Zakhar Kabluchko: Institut f\"ur Mathematische Stochastik,
Westf\"alische {Wilhelms-Uni\-ver\-sit\"at} M\"unster,
Orl\'eans-Ring 10,
48149 M\"unster, Germany}
\email{zakhar.kabluchko@uni-muenster.de}

\author{Christoph Th\"ale}
\address{Christoph Th\"ale: Fakult\"at f\"ur Mathematik,
Ruhr-Universit\"at Bochum, Universitätsstr. 150,
44780 Bochum, Germany}
\email{christoph.thaele@rub.de}

\title[Weyl random cones in high dimensions]{Random cones in high dimensions II: Weyl cones}



\begin{abstract}
We consider two models of random cones together with their duals. Let $Y_1,\dots,Y_n$ be independent and identically distributed random vectors in $\RR^d$ whose distribution satisfies some mild condition. The random cones $G_{n,d}^A$ and $G_{n,d}^B$ are defined as the positive hulls $\pos\{Y_1-Y_2,\dots,Y_{n-1}-Y_n\}$, respectively $\pos\{Y_1-Y_2,\dots,Y_{n-1}-Y_n,Y_n\}$, conditioned on the event that the respective positive hull is not equal to $\RR^d$. We prove limit theorems for various expected geometric functionals of these random cones, as $n$ and $d$ tend to infinity in a coordinated way. This includes limit theorems for the expected number of $k$-faces and the $k$-th conic quermassintegrals, as $n$, $d$ and sometimes also $k$ tend to infinity simultaneously. Moreover, we uncover a phase transition in high dimensions for the expected statistical dimension for both models of random cones.
\end{abstract}

\maketitle

\tableofcontents

\section{Introduction}
A polyhedral cone (just called cone in this paper for simplicity)  $C\subset\RR^d$, $d\in\NN$, is defined as an intersection of finitely many closed half-spaces whose bounding hyperplanes pass through the origin. The present paper deals with polyhedral cones whose bounding hyperplanes are chosen randomly. Two natural models for random cones, the so-called Cover-Efron and Donoho-Tanner random cones together with their dual cones were already treated in part I~\cite{GKT2020_HighDimension1} of this series of papers. Let us recall that the \textit{Cover-Efron random cone} in $\RR^d$, for independent random vectors $X_1,\dots,X_n$ taking values in $\RR^d$ and being identically distributed according to a symmetric density like the standard Gaussian distribution on $\RR^d$, is defined as the random positive hull
$$
\pos\{X_1,\dots,X_n\}:=\bigg\{\sum_{i=1}^n\lambda_iX_i:\lambda_1,\dots,\lambda_n\ge 0\bigg\},
$$
conditioned on the event that $\pos\{X_1,\dots,X_n\}\neq\RR^d$. These cones were first introduced and studied in the classical work of Cover and Efron~\cite{CoverEfron}. The \textit{Donoho-Tanner random cone} in $\RR^d$ is defined as the random cone $\pos\{X_1,\dots,X_n\}$ -- without any conditioning -- and was introduced by Donoho and Tanner~\cite{DonohoTanner}. In part I~\cite{GKT2020_HighDimension1} of this series, continuing and expanding the work of Hug and Schneider~\cite{HugSchneiderThresholdPhenomena}, we proved limit theorems for various expected combinatorial and geometric functionals, like the expected number of $k$-(dimensional) faces or the conic intrinsic volumes, for Cover-Efron and Donoho-Tanner random cones, and their duals, in high-dimensions, that is, in regimes where the number of vectors $n$ and the dimension $d$ tend to infinity simultaneously in a coordinated way. In particular, the papers~\cite{DonohoTanner,GKT2020_HighDimension1,HugSchneiderThresholdPhenomena,HugSchneiderThresholdPhenomena2} uncovered a number of high-dimensional threshold phenomena and phase transitions.

In the present part II we apply similar methods as in~\cite{GKT2020_HighDimension1,HugSchneiderThresholdPhenomena} to develop analogous limit theorems for two different classes of random cones, the so-called Weyl random cones of type $A$ and $B$. To define them, let $Y_1,Y_2,\dots$ be a sequence of independent random vectors in $\RR^d$ distributed according to a probability measure $\mu$ on $\RR^d$ that assigns measure zero to each affine hyperplane. Then, as in~\cite{GodlandKabluchko}, we consider the following two random cones, for the notion of the dual of a cone we refer to \eqref{eq:dualcone} below.
\begin{itemize}
\item[(i)] Let $G_{n,d}^A$ be the random cone whose distribution is that of $\pos\{Y_1-Y_2,\dots,Y_{n-1}-Y_n\}$, conditioned on the event that $\pos\{Y_1-Y_2,\dots,Y_{n-1}-Y_n\}\neq\RR^d$. Then the \textit{Weyl random cone $W_{n,d}^A:=(G_{n,d}^A)^\circ$ of type $A$} is the dual cone of $G_{n,d}^A$.
\item[(ii)] Let $G_{n,d}^B$ be the random cone whose distribution is that of $\pos\{Y_1-Y_2,\dots,Y_{n-1}-Y_n,Y_n\}$, again conditioned on the event that this cone is different from $\RR^d$. Then the \textit{Weyl random cone $W_{n,d}^B:=(G_{n,d}^B)^\circ$ of type $B$} is the dual cone of $G_{n,d}^B$.
\end{itemize}
In Section \ref{sec:Tess+Cones} we will introduce the Weyl random cones in a different, but equivalent, way as typical conical cells of a conical Weyl random tessellations of type $A$ and $B$, respectively.

The combinatorial and geometric properties of the Weyl random cones $W_{n,d}^A$ and $W_{n,d}^B$ are closely linked to characteristics of the Weyl chambers of the classical reflection arrangements ${\rm arr}(A_{n-1})$ and ${\rm arr}(B_n)$ of types $A_{n-1}$ and $B_n$ in $\RR^n$, respectively. The latter are given by
\begin{align*}
{\rm arr}(A_{n-1}) &:= \{(e_i-e_j)^\perp : 1\leq i<j\leq n\},\\
{\rm arr}(B_n) &:= \{(e_i-e_j)^\perp,(e_i+e_j)^\perp : 1\leq i<j\leq n\}\cup\{e_i^\perp:1\leq i\leq n\},
\end{align*}
where $e_1,\ldots,e_n$ is the standard orthonormal basis in $\RR^n$, and the Weyl chambers are the closed polyhedral cones into which the hyperplanes of these arrangements dissect the space. It has been shown in \cite{GodlandKabluchko} that formulas for the expected number of $k$-faces or the expected conic intrinsic volumes of the Weyl random cones $W_{n,d}^A$ and $W_{n,d}^B$ can be reduced to questions about the number of faces or chambers of the reflection arrangements ${\rm arr}(A_{n-1})$ and ${\rm arr}(B_n)$ that are intersected by a random linear subspace. The answers to these questions in turn can be expressed in terms of the characteristic polynomials of the two hyperplane arrangements. 
The coefficients of these polynomials are known as the Stirling numbers of first kind $A(n,k)$ and their B-analogues $B(n,k)$, $k\in\{0,1,\ldots,n\}$. Both sequences induce a probability distribution on the discrete set $\{0,1,\ldots,n\}$ and limit theorems for the random variables distributed according to these laws play an essential role in this paper. We remark that the resulting distributions are convolutions of Bernoulli distributions. However, while in the setting of Cover-Efron and similar cones studied in~\cite{DonohoTanner,GKT2020_HighDimension1,HugSchneiderThresholdPhenomena,HugSchneiderThresholdPhenomena2} the parameters of these distributions were all the same, namely $1/2$, this is no more the case in the present paper, which makes the probabilistic arguments more involved.

To illustrate the types of results we develop in this paper, let us present two representative examples. To treat both types of random cones simultaneously, here and throughout the paper let $\blackdiamond$ be one of the symbols $A$ or $B$ and $W_{n,d}^\blackdiamond$, respectively $G_{n,d}^\blackdiamond$,  be the Weyl random cone and its dual. Also, put $\sigma_A:=1$ and $\sigma_B:=1/2$. We are interested in the expected number of $k$-faces $\bE f_k(G_{n,d}^\blackdiamond)$ of the random cones $G_{n,d}^\blackdiamond$, as $n\to\infty$. We consider the situation where both $k$ and $d$ depend on $n$ in such a way that $d=n-\sigma_\blackdiamond x\log n+o(\log n)$, as $n\to\infty$, and $x>0$ is constant. Theorem \ref{theorem:phase_trans_faces_dual_Weyl_A} below uncovers the following phase transition. If $x>1$, we show that
\begin{align*}
	\lim_{n\to\infty}\frac{\bE f_k(G_{n,d}^\blackdiamond)}{\binom{n+1-2\sigma_\blackdiamond}k}=
	\begin{cases}
		1				&: k=o(n),\\
		(1-\alpha)^{x-1}			&: k=\alpha n+o(n)\text{, }\alpha\in(0,1),\\
		0				&: k=n+o(n),
	\end{cases}
\end{align*}
while for $x\in(0,1)$ it holds that
\begin{align*}
	\lim_{n\to\infty}\frac{\bE f_k(G_{n,d}^\blackdiamond)}{\binom{n+1-2\sigma_\blackdiamond}k}=
	\begin{cases}
		1				&: k=n-\exp\{c\log n+o(\log n)\} \text{ with } c\in(x,1),\\
		1-\Phi(\alpha)				&: k=n-\exp\{{\sigma_{\blackdiamond}^{-1}(n-d-\alpha\sqrt{x\sigma_\blackdiamond\log n})}	\},\; \alpha\in\RR				,\\
		0				&: k=n-\exp\{c\log n+o(\log n)\} \text{ with } c\in(0,x),
	\end{cases}
\end{align*}
where $\Phi$ denotes the distribution function of a standard normal distributed random variable.
Furthermore, Theorem \ref{theorem:large_dev_f_k_dual_Weyl_A} yields a kind of large deviations principle for $\bE f_k(G_{n,d}^\blackdiamond)$. Namely, if we assume in addition that $k=n-\exp\{c\log n+o(\log n)\}$, then for all $x>0$ and $c\in(0,1)$ we prove that
\begin{align*}
	\lim_{n\to\infty}\frac{1}{\log n}\log\frac{\bE f_k(G_{n,d}^\blackdiamond)}{\binom{n+1-2\sigma_\blackdiamond}k}=
	\begin{cases}
		x\log c-c+1	&:x>1,\\
		x-x\log x+x\log c-c	&: x\in(0,1),c\in(0,x),\\
		0						&:x\in(0,1),c\in(x,1).
	\end{cases}
\end{align*}
Similar results are obtained for the so-called conic intrinsic volumes and the conic quermassintegrals of the random cones $G_{n,d}^\blackdiamond$ and $W_{n,d}^\blackdiamond$ as well.

\medspace

This paper is structured as follows. Section~\ref{sec:prelimiaries} introduces some notation, contains the formal definitions of various geometric functionals for convex cones and collects limit theorems for Stirling numbers of the first kind and their B-analogues. In Section~\ref{sec:Tess+Cones} we formally introduce the Weyl tessellations which are used to define the cones $G_{n,d}^A$  and $G_{n,d}^B$, or rather their duals, the Weyl random cones. We also rephrase there a number of known results on which our work is based. In Sections~\ref{sec:limit_theorems_Weyl_faces},~\ref{sec:limit_intr_vol_quermass} and~\ref{sec:limit_stat_dim} we state and prove the limit theorems for the expectations of various combinatorial and geometric functionals of the cones $G_{n,d}^A$ and $G_{n,d}^B$, as well as their dual cones.

\section{Preliminaries}\label{sec:prelimiaries}

\subsection{Notation}
In this paper $N(0,1)$ denotes a standard normal random variable and $\Phi$ the distribution function of $N(0,1)$, that is,
\begin{align*}
\Phi(x)=\int_{-\infty}^x\frac{1}{\sqrt{2\pi}}e^{-t^2/2} \dint t,\qquad x\in\RR.
\end{align*}
The almost sure convergence of sequences of random variables indexed by $n$ is denoted by $\toas$, the convergence in probability by $\toprobab$, while the convergence in distribution is denoted by $\todistr$. We slightly abuse notation and write $X_n\todistr X$ or $X_n\todistr \mu$ to indicate that the sequence of random variables $X_n$ converges in distribution to a random variable $X$ with law $\mu$.

Two sequences $(a_n)_{n\ge 0}$ and $(b_n)_{n\ge 0}$ of real numbers are asymptotically equivalent, as $n\to\infty$, if $a_n/b_n$ converges to $1$, as $n\to\infty$. This is denoted by $a_n\tosim b_n$, or $a_n\sim b_n$ if the index tending to infinity is clear from the context. Moreover, we use the well-known Landau-notation $o(a_n)$ for a sequence that tends to $0$, as $n\to\infty$, after being divided by $a_n$. Similarly, we say $b_n=O(a_n)$, as $n\to\infty$, if
$$
\limsup_{n\to\infty}\left|\frac{b_n}{a_n}\right|<\infty.
$$

\subsection{Limit theorems for Stirling numbers of the first kind and their \texorpdfstring{$\boldsymbol{B}$}{B}-analogues}\label{section:limit_theorem_stirling}

Throughout this paper, the arguments for most results are based on various limit theorems for Stirling numbers of the first kind and their so-called $B$-analogues. The Stirling number of the first kind $A(n,k)=\stirling{n}{k}$ is defined as the number of permutations of the set $\{1,\dots,n\}$, $n\in\NN$, having exactly $k\in\{1,\ldots,n\}$ cycles. Equivalently, the Stirling numbers are the coefficients of the polynomial
\begin{align}\label{eq:def_stirling1}
t(t+1)\cdot\ldots\cdot(t+n-1)=\sum_{k=1}^nA(n,k)t^k
\end{align}
and, by convention, we put $A(n,k)=0$ for $k\notin\{1,\dots,n\}$. Similarly, the $B$-analogues of the Stirling number of first kind are denoted by $B(n,k)$ and can be defined as the coefficients of the polynomial
\begin{align}\label{eq:def_stirling1b}
(t+1)(t+3)\cdot\ldots\cdot(t+2n-1)=\sum_{k=0}^nB(n,k)t^k;
\end{align}
again we put $B(n,k)=0$ for $k\notin\{0,\dots,n\}$.

Next, we introduce the following notation. For $\blackdiamond\in\{A,B\}$ we define
$$
\sigma_\blackdiamond:=\begin{cases}
1 &: \blackdiamond=A\\
{1\over 2} &: \blackdiamond=B.
\end{cases}
$$
Then, we can introduce the random variable $S_n^\blackdiamond$, $n\in\NN$, as the sum
$$
S_n^\blackdiamond := \sum\limits_{k=1}^n\Bern\Big({\sigma_\blackdiamond\over k}\Big),
$$
where $\{\Bern(\sigma_\blackdiamond/k):k\geq 1\}$ is a sequence of independent Bernoulli random variables with parameters as indicated in brackets. More precisely, for $p\in[0,1]$, the distribution of $\Bern(p)$ is given by
$$
\bP[\Bern(p)=1] = p\qquad\text{and}\qquad\bP[\Bern(p)=0] = 1-p.
$$
We note that $S_n^\blackdiamond$ is defined in such a way that its probability mass function is given by the (normalized) Stirling numbers (if $\blackdiamond=A$) or their (normalized) $B$-analogues (if $\blackdiamond=B$). That is,
$$
\bP[S_n^\blackdiamond=k] = {\blackdiamond(n,k)\over n!}\,\sigma_\blackdiamond^n,\qquad k\in\{0,1,\ldots,n\}.
$$
Indeed, this easily follows from the product structure of the generating function of $S_n^\blackdiamond$, for example.
From \cite[Example 2.1.3]{Feray_modphi2016} for case $\blackdiamond=A$ and \cite[Lemma~5.3]{KabluchkoVysotskyZaporozhets} for case $\blackdiamond=B$ it is known that the sequence of random variables $S_n^\blackdiamond$ satisfies the following mod-Poisson convergence:
\begin{equation}\label{eq:ModPoisson}
\lim_{n\to\infty}{\bE[e^{zS_n^\blackdiamond}]\over e^{\sigma_\blackdiamond\log n(e^z-1)}} = \Psi_\blackdiamond(z):={1\over\Gamma(\sigma_\blackdiamond(e^z+2(1-\sigma_\blackdiamond)))},\qquad\qquad z\in\CC,
\end{equation}
where we recall that $e^{\sigma_\blackdiamond\log n(e^z-1)}$ is the moment generating function of a Poisson random variable with parameter $\sigma_\blackdiamond\log n$. In fact, $\Psi_A(z)=1/\Gamma(e^z)$ as in \cite[Example 2.1.3]{Feray_modphi2016}, while $\Psi_B(z)=1/\Gamma({e^z+1\over 2})$. In the latter case, an application of Legendre's duplication formula yields that $\Psi_B$ can be rewritten as $\Psi_B(z)={2^{e^z}\Gamma({e^z\over 2})\over 2\sqrt{\pi}\Gamma(e^z)}$, which is the form of $\Psi_B$ used in \cite[Lemma~5.3]{KabluchkoVysotskyZaporozhets}.

From the mod-Poisson convergence \eqref{eq:ModPoisson} a number of probabilistic limit theorems and estimates follow, see \cite[Theorem~3.3.1]{Feray_modphi2016} for (i), \cite[Theorem 5.2]{KabluchkoVysotskyZaporozhets} for (ii), \cite[Example~3.2.6]{Feray_modphi2016} for claim (iv) and \cite[Theorem~3.2.2]{Feray_modphi2016} for (v). 
\begin{itemize}
\item[(i)] We have the central limit theorem
\begin{align}\label{eq:CLT_Stirling}
{S_n^\blackdiamond-\sigma_\blackdiamond\log n\over\sqrt{\sigma_\blackdiamond\log n}} \underset{n\to\infty}{\overset{d}{\longrightarrow}} N(0,1),
\end{align}
where $N(0,1)$ is a standard Gaussian random variable.

\item[(ii)] We have the following central limit type result for the random variables $S_n^\blackdiamond$:
\begin{align}\label{eq:CLT_stirling_kind_of}
\lim_{n\to\infty}2\sum_{\ell=1,3,\ldots}\bP\big[S_n^\blackdiamond = \sigma_\blackdiamond\log n+v_n\sqrt{\sigma_\blackdiamond\log n}-\ell\big] = \Phi(v),
\end{align}
where $v_n\to v$ is any convergent real sequence such that $\sigma_\blackdiamond\log n+v_n\sqrt{\sigma_\blackdiamond\log n}\in\textcolor{green}{\NN}$ for all $n\in\mathbb N$.
\item[(iii)] We have the following weak law of large numbers
\begin{align}\label{eq:weakLLN}
\frac{S_n^\blackdiamond}{\sigma_\blackdiamond\log n}\toprobab 1,
\end{align}
which follows directly from the central limit theorem~\eqref{eq:CLT_Stirling}.

\item[(iv)]In what follows, we let $(z_n)_{n\in\NN}$ be a sequence satisfying $z_n\to z\in\RR$, as $n\to\infty$, and $z_n\sigma_\blackdiamond\log n\in\NN$ for each $n\in\NN$.  As $n\to\infty$, we have the asymptotic relationships
\begin{equation*}
\bP\big[S_n^\blackdiamond = z_n\sigma_\blackdiamond\log n\big] = {n^{-(z_n\log z_n-z_n+1)}\over\sqrt{2\pi z\log n}}\,\Psi_\blackdiamond(\log z)\Big(1+O\Big({1\over\log n}\Big)\Big)
\end{equation*}
for $z>0$ and
\begin{equation}\label{eq:asympt_S_n>=}
\bP\big[S_n^\blackdiamond \geq z_n\sigma_\blackdiamond\log n\big] = {n^{-(z_n\log z_n-z_n+1)}\over\sqrt{2\pi z\log n}}\,{z\over z-1}\Psi_\blackdiamond(\log z)\Big(1+O\Big({1\over\log n}\Big)\Big)
\end{equation}
for $z>1$.

\item[(v)] For $\ell\in\ZZ$ and $z>0$ we have that
\begin{equation}\label{eq:AsymptoticsGEQwithZ_noSum}
\bP\big[S_n^\blackdiamond=z_n\sigma_\blackdiamond\log n+\ell\big]\underset{n\to\infty}{\sim} {n^{-(z_n\log z_n-z_n+1)}\over\sqrt{2\pi z\log n}}\,\Psi_\blackdiamond(\log z)\,z^{-\sigma_\blackdiamond \ell}.
\end{equation}
\end{itemize}

We note that from (iv) it follows that, for $z>1$,
\begin{align}\label{eq:AsymptoticsGEQwithZ>1}
\sum_{\ell=1,3,\ldots}\bP\big[S_n^\blackdiamond = z_n\sigma_\blackdiamond\log n+\ell\big]\underset{n\to\infty}{\sim}{n^{-(z_n\log z_n-z_n+1)}\over\sqrt{2\pi z\log n}}\,\Psi_\blackdiamond(\log z)\,{z^{\sigma_\blackdiamond}\over z^{2\sigma_\blackdiamond}-1}\underset{n\to\infty}{\longrightarrow} 0,
\end{align}
while for $z\in(0,1)$ it holds that
\begin{align}\label{eq:AsymptoticsGEQwithZ<1}
\sum_{\ell=1,3,\ldots}\bP\big[S_n^\blackdiamond = z_n\sigma_\blackdiamond\log n+\ell\big]\underset{n\to\infty}{\sim}{1\over 2}-{n^{-(z_n\log z_n-z_n+1)}\over\sqrt{2\pi z\log n}}\,\Psi_\blackdiamond(\log z)\,{z^{\sigma_\blackdiamond}\over 1- z^{2\sigma_\blackdiamond}}\underset{n\to\infty}{\longrightarrow} {1\over 2}.
\end{align}
Indeed, the first relation is a consequence of the dominated convergence theorem, whose application can formally be justified by the fact that the sequence of Stirling numbers and the sequence of their $B$-analogues are log-concave (which in turn is a consequence of recurrence relations for $A(n,\ell)$ and $B(n,\ell)$). The second claim follows the same way, additionally using that
\begin{align}\label{eq:stirling_numbers_rest}
\blackdiamond(n,0)+\blackdiamond(n,2)+\ldots=\blackdiamond(n,1)+\blackdiamond(n,3)+\ldots={n!\over 2\sigma_\blackdiamond^n},
\end{align}
for all $n\geq 2$,  a relation which arises from inserting $t=\pm 1$ into~\eqref{eq:def_stirling1} and~\eqref{eq:def_stirling1b}.

\subsection{Convex cones and their intrinsic volumes}

The positive hull of a set $M\subset\RR^d$ is defined as
\begin{align*}
\pos\, M:=\Big\{\sum_{i=1}^m\lambda_it_i:\,m\in\NN,t_1,\dots,t_m\in M,\lambda_1,\dots,\lambda_m\ge 0\Big\}.
\end{align*}
A convex set $C\subset\RR^d$ is called a \textit{convex polyhedral cone} (or just a cone) if it is a positive hull of finitely many vectors. 

A linear hyperplane $H$ supports a cone $C$, provided that $C$ is contained in one of the two closed half-spaces determined by $H$. If $H$ is a supporting hyperplane of $C$ then $C\cap H$ is called a \textit{face} of $C$. We say that a face has dimension $k\in\{0,1,\ldots,d\}$ if its linear hull is a $k$-dimensional linear subspace of $\RR^d$. The set of all $k$-dimensional faces ($k$-faces for short) of a cone $C$ is denoted by $\cF_k(C)$ and we let $f_k(C):=\#\cF_k(C)$ be the number of $k$-faces of $C$.

The \textit{dual cone $C^\circ$} of a cone $C\subset\RR^d$ is given as
\begin{align}\label{eq:dualcone}
C^\circ:=\{v\in\RR^d:\langle v,x\rangle\le 0\ \text{for all}\ x\in C\},
\end{align}
where $\langle\,\cdot\,,\,\cdot\,\rangle$ is the usual Euclidean scalar product in $\RR^d$.

We now introduce two series of geometric quantities associated with a convex cone $C\subset\RR^d$. For $k\in\{0,1,\ldots,d\}$ the \textit{k-th conic intrinsic volume} $\upsilon_k(C)$ of $C$ is defined as
\begin{align*}
\upsilon_k(C):=\sum_{F\in\cF_k(C)}\bP[\Pi_C(g)\in\relint\, F],
\end{align*}
where $g$ is a standard Gaussian random vector in $\RR^d$, $\relint \,F$ denotes the relative interior of $F$ and $\Pi_C(x)$, $x\in\RR^d$, is the point $y\in C$ minimizing the Euclidean distance to $x$, the so-called metric projection of $x$ onto $C$. An equivalent definition of the conic intrinsic volumes using the spherical Steiner formula can be found in~\cite[Section~6.5]{SW}. For further properties we refer to in~\cite[Section~2.2]{AmelunxenLotz}, and also~\cite[Section~6.5]{SW}.

Next, for a cone $C\subset\RR^d$ that is not a linear subspace, the \textit{$k$-th conic quermassintegral} of $C$, $k\in\{0,1,\dots,d\}$, is defined as
\begin{align*}
U_k(C):=\frac 12\bP[C\cap W_{d-k}\neq\{0\}],
\end{align*}
where $W_{d-k}$ is a uniformly distributed $(d-k)$-subspace random in the Grassmannian $G(d,d-k)$ of all $(d-k)$-linear subspaces of $\RR^d$.
For a $j$-dimensional linear subspace $L_j\subset\RR^d$, we put
\begin{align*}
U_k(L_j) :=
\begin{cases}
1		&: j-k>0\text{ and odd}\\
0		&: j-k\le 0\text{ or even}.
\end{cases}
\end{align*}
We remark that if $C$ is not a linear subspace, the quantity $2U_k(C)$ is also known as the \textit{$k$-th Grassmann angle} of $C$ and was introduced by Grünbaum~\cite{gruenbaum_grass_angles}. For further properties of the conic quermassintegrals see also~\cite[Section~2]{HugSchneiderConicalTessellations}.

\section{Weyl tessellations, Weyl random cones and their duals}\label{sec:Tess+Cones}

Fix a probability measure $\mu$ on $\RR^d$, which satisfies $\mu(H)=0$ for each (linear and affine) hyperplane $H$ in $\RR^d$, and let $Y_1,Y_2,\ldots$ be a sequence of independent random vectors with distribution $\mu$ (these assumptions can be slightly weakened, see~\cite{GodlandKabluchko}). Also, let $n\geq d+1$. The \textit{Weyl tessellation of type $A$} is the conical tessellation of $\RR^d$ induced by the $n(n-1)/2$ random hyperplanes
\begin{align*}
(Y_i-Y_j)^\perp,\qquad 1\le i<j\le n,
\end{align*}
while by the \textit{Weyl tessellation of type $B$} we understand the conical random tessellation that is induced by the $n^2$ hyperplanes
\begin{align*}
	(Y_i-Y_j)^\perp,&\qquad 1\le i<j\le n,\\
	(Y_i+Y_j)^\perp,&\qquad 1\le i<j\le n,\\
	Y_i^\perp,&\qquad 1\le i\le n.
\end{align*}
A realization of both Weyl tessellations is shown in Figure~\ref{pic:1}.
\begin{figure}[!t]
\centering
\includegraphics[scale=0.33]{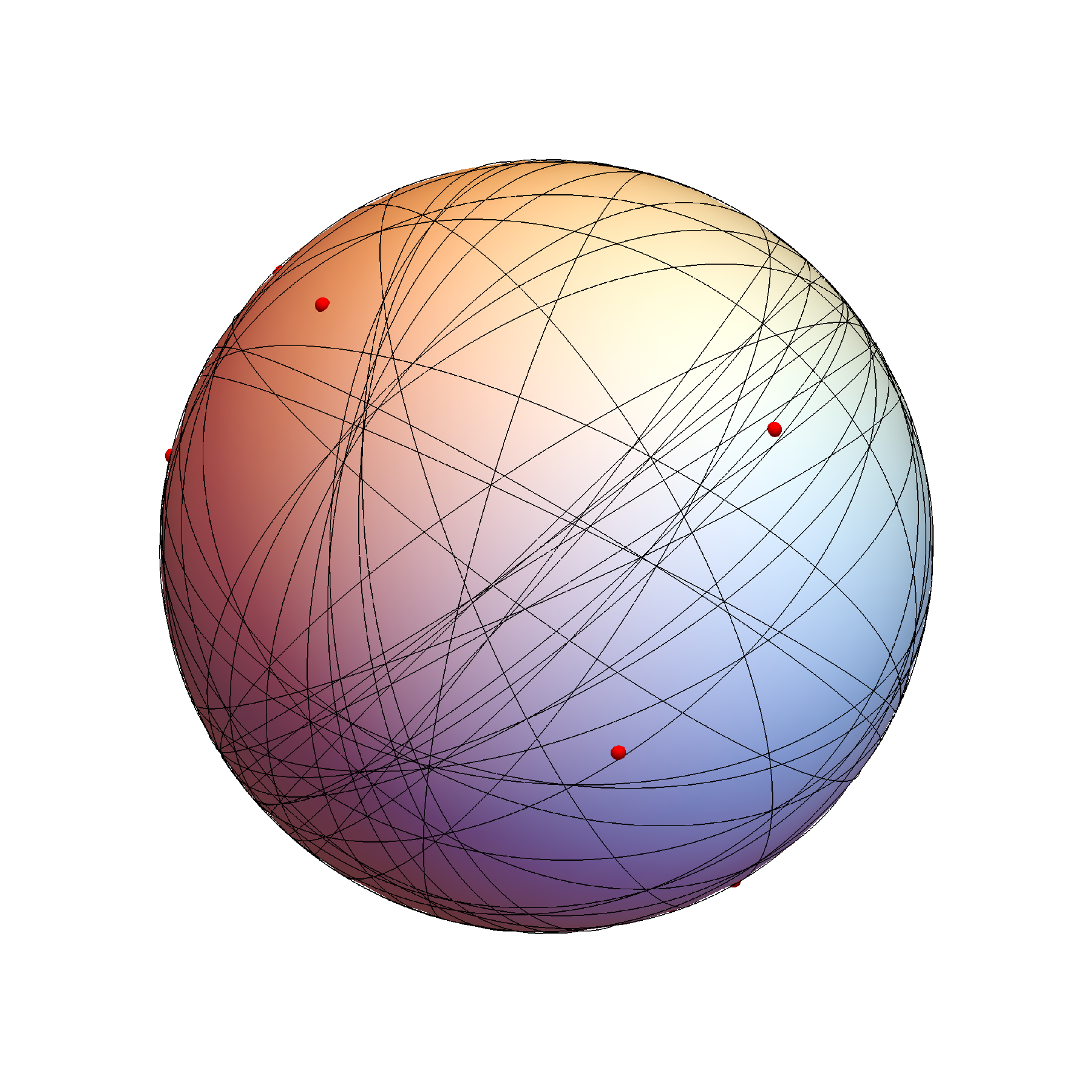}
\includegraphics[scale=0.33]{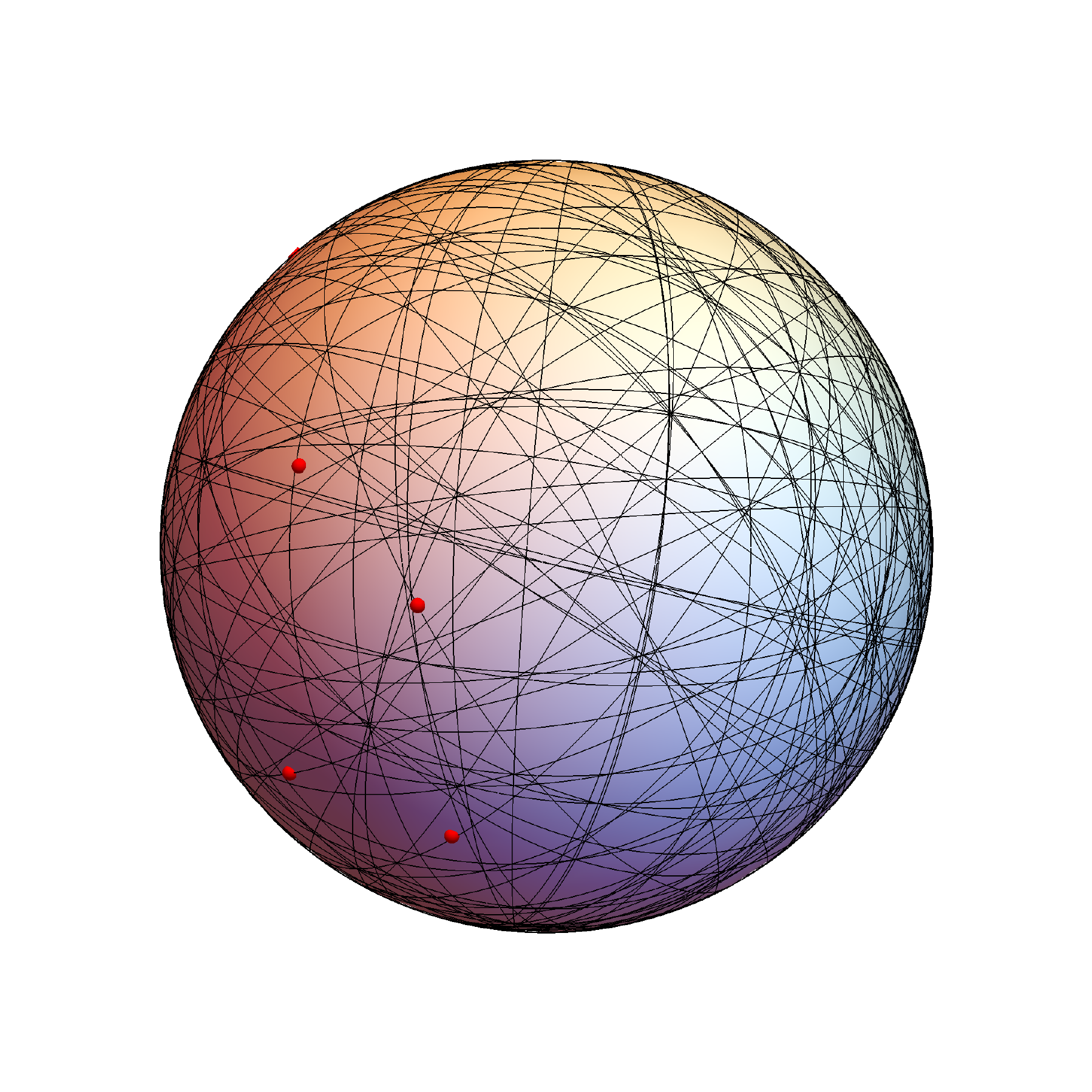}
\caption{
Weyl tessellations of type $A$ (left) and $B$ (right) intersected with the unit sphere in $\RR^3$ generated by $n=10$ points.
The vectors $Y_1,\ldots,Y_n$ (red points) were drawn independently and uniformly from the unit sphere.
}\label{pic:1}
\end{figure}
From \cite[Corollaries 3.4 and 4.4]{GodlandKabluchko} it follows that the Weyl tessellation of type $\blackdiamond$ with $\blackdiamond\in\{A,B\}$ almost surely consists of
$$
D^\blackdiamond(n,d) := 2[\blackdiamond(n,n-d+1)+\blackdiamond(n,n-d+3)+\ldots]
$$
random convex cones, independently of the choice of $\mu$. The random \textit{Weyl cone $W_{n,d}^\blackdiamond$ of type $\blackdiamond$} can now be defined as a random cone chosen uniformly at random from the collection of $D^\blackdiamond(n,d)$ cones in the Weyl tessellation of type $\blackdiamond\in\{A,B\}$. In what follows we will denote by $G_{n,d}^\blackdiamond:=(W_{n,d}^\blackdiamond)^\circ$ the dual of the Weyl cone.
In~\cite{GodlandKabluchko} it was shown that $G_{n,d}^\blackdiamond$ can equivalently be defined as the random cone $\pos\{Y_1-Y_2,\dots,Y_{n-1}-Y_n\}$ (in the case $\blackdiamond=A$) or as the random cone $\pos\{Y_1-Y_2,\dots,Y_{n-1}-Y_n,Y_n\}$ (in the case $\blackdiamond=B$) conditioned on the event that the respective positive hull is not equal to $\RR^d$.

We now rephrase, in a unified way, the explicit formulas for the expected conic intrinsic volumes, the expected the conic quermassintegrals and the expected face numbers of Weyl random cones and their duals:
\begin{eqnarray}
	\bE[\nu_k(W_{n,d}^\blackdiamond)] &=& \begin{cases}
		{\blackdiamond(n,n-d+k)\over D^\blackdiamond(n,d)} &: k\in\{1,\ldots,d\}\\[2mm]
		{D^\blackdiamond(n,d)-D^\blackdiamond(n,d-1)\over 2D^\blackdiamond(n,d)} &: k=0,
	\end{cases}\label{eq:IntVolWeyl}\\
	\bE[\nu_k(G_{n,d}^\blackdiamond)] &=&  \begin{cases}
		{\blackdiamond(n,n-k)\over D^\blackdiamond(n,d)} &: k\in\{0,\ldots,d-1\}\\[2mm]
		{D^\blackdiamond(n,d)-D^\blackdiamond(n,d-1)\over 2D^\blackdiamond(n,d)} &: k=d,
	\end{cases}\label{eq:IntVolWeylDual}
\end{eqnarray}
see \cite[Corollary~1.10]{GodlandKabluchko},
\begin{eqnarray}
	\bE[U_k(W_{n,d}^\blackdiamond)] &=& {D^\blackdiamond(n,d-k)\over 2D^\blackdiamond(n,d)},\qquad\qquad\quad\,\, k\in\{0,1,\ldots,d-1\},\label{eq:QuermassIntWeyl}\\
	\bE[U_k(G_{n,d}^\blackdiamond)] &=& {D^\blackdiamond(n,d)-D^\blackdiamond(n,k)\over 2D^\blackdiamond(n,d)},\qquad k\in\{1,\ldots,d\},\label{eq:QuermassIntWeylDual}
\end{eqnarray}
see \cite[Corollary~1.9]{GodlandKabluchko}, and
\begin{eqnarray}
	\bE[f_k(W_{n,d}^\blackdiamond)] &=& {{n+1-2\sigma_\blackdiamond\choose d-k}D^\blackdiamond(n-d+k,k)\over \sigma_\blackdiamond^{d-k}D^\blackdiamond(n,d)}{n!\over(n-d+k)!},\qquad k\in\{1,\ldots,d\},\label{eq:FaceNumberWeyl}\\
	\bE[f_k(G_{n,d}^\blackdiamond)] &=& {{n+1-2\sigma_\blackdiamond\choose k}D^\blackdiamond(n-k,d-k)\over \sigma_\blackdiamond^{k}D^\blackdiamond(n,d)}{n!\over(n-k)!},\qquad\qquad\! k\in\{0,1,\ldots,d-1\},\label{eq:FaceNumberWeylDual}
\end{eqnarray}
see \cite[Theorems~1.5 and~1.7]{GodlandKabluchko}. We recall that $\sigma_\blackdiamond=1$ for $\blackdiamond=A$ and $\sigma_\blackdiamond=1/2$ if $\blackdiamond=B$.

\section{Limit theorems for the expected number of faces}\label{sec:limit_theorems_Weyl_faces}

In the remaining sections of this paper, we develop various limit theorems for the expectations of various geometric functionals of the Weyl random cones $W_{n,d}^\blackdiamond$ and their duals $G_{n,d}^\blackdiamond$ for both types $\blackdiamond=A$ and $\blackdiamond=B$. In particular, in the present section we consider the expected number of $k$-faces, in the next section the conic intrinsic volumes as well as the conic quermassintegrals, and finally in Section \ref{sec:limit_stat_dim}, the statistical dimension of the above random cones. In each of these results, we consider a regime where $n\to\infty$ and, as a function of $n $, also $d=d(n)\to\infty$ in a coordinated way.

In the present section, we want to understand the asymptotic behaviour of the expected number of $k$-faces of the random cones introduced in Section \ref{sec:Tess+Cones} in high dimensions, that is, as $n$, the dimension $d$, and in some cases also $k$ tend to infinity simultaneously in a coordinated way. The next theorem can be considered as an analogue to~\cite[Theorem~7]{HugSchneiderThresholdPhenomena} in the setting of Weyl cones and uncovers threshold phenomena in $k$ for $\bE f_k(G_{n,d}^\blackdiamond)$ divided by $\binom{n+1-2\sigma_\blackdiamond}{k}$. Note that we divide by this binomial coefficient since any $k$-face of $G_{n,d}^\blackdiamond$ almost surely is the positive hull of $k$ vectors from $Y_1-Y_2,\dots,Y_{n-1}-Y_n$ (for $\blackdiamond=A$) or from $Y_1-Y_2,\dots,Y_{n-1}-Y_n,Y_n$ (for $\blackdiamond=B$) and there are $\binom{n+1-2\sigma_\blackdiamond}{k}$ possible choices.

\begin{theorem}\label{theorem:phase_trans_faces_dual_Weyl_A}
Let $\blackdiamond\in\{A,B\}$ and $G_{n,d}^\blackdiamond$ be the dual of the Weyl random cone $W_{n,d}^\blackdiamond$. Consider $k=k(n)$ and $d=d(n)$ and suppose that
$$
d=n-\sigma_\blackdiamond x\log n+o(\log n),\qquad\text{as }n\to\infty.
$$
In the case $x>1$, we have that
\begin{align*}
\lim_{n\to\infty}\frac{\bE f_k(G_{n,d}^\blackdiamond)}{\binom{n+1-2\sigma_\blackdiamond}k}=
\begin{cases}
1				&: k=o(n),\\
(1-\alpha)^{x-1}			&: k=\alpha n+o(n)\text{, }\alpha\in(0,1),\\
0				&: k=n+o(n),
\end{cases}
\end{align*}
while for $x\in(0,1)$ it holds that
\begin{align*}
\lim_{n\to\infty}\frac{\bE f_k(G_{n,d}^\blackdiamond)}{\binom{n+1-2\sigma_\blackdiamond}k}=
\begin{cases}
1				&: k=n-\exp\{c\log n+o(\log n)\} \text{ with } c\in(x,1),\\
1-\Phi(\alpha)				&: k=n-\exp\{{\sigma_{\blackdiamond}^{-1}(n-d-\alpha\sqrt{x\sigma_\blackdiamond\log n})}	\},\; \alpha\in\RR				,\\
0				&: k=n-\exp\{c\log n+o(\log n)\} \text{ with } c\in(0,x).
\end{cases}
\end{align*}
\end{theorem}

\begin{proof}
In the given regime, we have that $d=n-\sigma_\blackdiamond x_n\log n$ for a sequence $(x_n)_{n\in\NN}$ such that $\lim_{n\to\infty}x_n=x>0$ and $\sigma_\blackdiamond x_n\log n\in\NN$ for each $n\in\NN$. At first, we use~\eqref{eq:FaceNumberWeylDual} and define $y_n:=(x_n\log n)/\log(n-k)$ to obtain
\begin{align}
\frac{\bE f_k(G_{n,d}^\blackdiamond)}{\binom{n+1-2\sigma_\blackdiamond}k}
&	=\frac{D^\blackdiamond(n-k,d-k)}{D^\blackdiamond(n,d)}\frac{n!}{\sigma_\blackdiamond^{k}(n-k)!}\notag\\
&=\frac{n!}{\sigma_\blackdiamond^{k}(n-k)!}{\sum\limits_{\ell=1,3,\ldots}\blackdiamond(n-k,n-d+\ell)\over\sum\limits_{\ell=1,3,\ldots}\blackdiamond(n,n-d+\ell)}\notag\\
&={\sum\limits_{\ell=1,3,\ldots}{\blackdiamond(n-k,n-d+\ell)\over(n-k)!}\sigma_\blackdiamond^{n-k}\over\sum\limits_{\ell=1,3,\ldots}{\blackdiamond(n,n-d+\ell)\over n!}\sigma_\blackdiamond^n}\notag\\
&	=\frac{\sum\limits_{\ell=1,3,\dots}\bP[S_{n-k}^\blackdiamond=n-d+\ell]}{\sum\limits_{\ell=1,3,\dots}\bP[S_n^\blackdiamond=n-d+\ell]}\label{eq:face_numbers_2}\\
&	=\frac{\sum\limits_{\ell=1,3,\dots}\bP[S_{n-k}^\blackdiamond=y_n\sigma_\blackdiamond\log(n-k)+\ell]}{\sum\limits_{\ell=1,3,\dots}\bP[S_n^\blackdiamond=x_n\sigma_\blackdiamond\log n+\ell]},\label{eq:face_numbers_1}
\end{align}
where $S_n^\blackdiamond$ (and similarly $S_{n-k}^\blackdiamond$) is the random variable with distribution $\bP[S_n^\blackdiamond=k]={\blackdiamond(n,k)\over n!}\sigma_\blackdiamond^n$, $k\in\{0,1,\ldots,n\}$ as considered in Section \ref{section:limit_theorem_stirling}.

\vspace*{2mm}
\noindent
\textsc{Case 1.}
Suppose first that $x>1$. In all regimes for $k$ it holds that $k=\alpha n+o(n)$ for a suitable $\alpha\in[0,1]$. Moreover, for $\alpha\in[0,1)$, the sequence $y_n$ converges to $x$, as $n\to\infty$. Inserting the asymptotic relation~\eqref{eq:AsymptoticsGEQwithZ>1} into both the numerator and the denominator of~\eqref{eq:face_numbers_1} yields
\begin{align}\label{eq:asym_faces_dual_Weyl_A}
\frac{\bE f_k(G_{n,d}^\blackdiamond)}{\binom{n+1-2\sigma_\blackdiamond}k}
&	\tosim \frac{\sqrt{\log n}}{\sqrt{\log(n-k)}}\frac{n}{n-k}\frac{(n-k)^{-(y_n\log y_n-y_n)}}{n^{-(x_n\log x_n-x_n)}},
\end{align}
independently of the choice of $\blackdiamond$. Recalling that $y_n=x_n\log n/\log(n-k)$ we see that
\begin{align*}
(n-k)^{-(y_n\log y_n-y_n)}
&	=\exp\left\{-\log(n-k)(y_n\log y_n-y_n)\right\}\\
&	=\exp\left\{-\log n\left(x_n\log y_n-x_n\right)\right\}\\
&	=n^{-(x_n\log y_n-x_n)},
\end{align*}
which implies
\begin{equation}\label{eq:EquivalenceForN}
\begin{split}	
\frac{(n-k)^{-(y_n\log y_n-y_n)}}{n^{-(x_n\log x_n-x_n)}} &= {n^{-(x_n\log y_n-x_n)}\over n^{-(x_n\log x_n-x_n)}}=n^{-x_n(\log y_n-\log x_n)}\\
&=n^{-x_n\log {y_n\over x_n}}=n^{-x_n\log{\log n\over\log(n-k)}}.
\end{split}
\end{equation}
Suppose now that $k=o(d)=o(n)$, in which case $n-k=n(1+o(1))$. Inserting this into \eqref{eq:EquivalenceForN} and the result in turn into \eqref{eq:asym_faces_dual_Weyl_A} we obtain
\begin{align*}
\frac{\bE f_k(G_{n,d}^\blackdiamond)}{\binom{n+1-2\sigma_\blackdiamond}k}
&	\tosim {\sqrt{\log n}\over\sqrt{\log (n(1+o(1)))}}{n\over n(1+o(1))}n^{-x_n\log{\log n\over\log(n(1+o(1)))}}\ton 1.
\end{align*}

Next, we assume that $k=\alpha d-o(d)=\alpha n-o(n)$ for some $\alpha\in(0,1)$. Then
\begin{align*}
	\frac{\sqrt{\log n}}{\sqrt{\log(n-k)}} &= \frac{\sqrt{\log n}}{\sqrt{\log(1-\alpha)+\log n+o(1)}} \ton 1,\\
	{n\over n-k} &= {n\over (1-\alpha)n+o(n)} \ton {1\over 1-\alpha}.
\end{align*}
To evaluate the last fraction in \eqref{eq:asym_faces_dual_Weyl_A} we use \eqref{eq:EquivalenceForN} to see that
$$
\frac{(n-k)^{-(y_n\log y_n-y_n)}}{n^{-(x_n\log x_n-x_n)}} = n^{-x_n\log{\log n\over\log(n-k)}} = n^{-x_n(\log\log n-\log\log(n-k))}.
$$
Now, using that
\begin{align*}
	\log\log(y+a) &= \log\log y+{a\over y\log y}+o\Big({a\over y\log y}\Big),\\
	\log(y+a) &= \log y+ \frac ay+o\Big(\frac ay\Big),
\end{align*}
as $y\to\infty$ for bounded $a=a(y)$, we see that, as $n\to\infty$,
\begin{align*}
\log\log n-\log\log(n-k) &\tosim \log\log n-\log\log (1-\alpha)n \\
&\hspace*{2.2mm}=\log\log n-\log[\log (1-\alpha)+\log n] \\
&\tosim -{\log(1-\alpha)\over\log n}.
\end{align*}
As a consequence, we have that
$$
\frac{(n-k)^{-(y_n\log y_n-y_n)}}{n^{-(x_n\log x_n-x_n)}} \tosim e^{-x_n(\log n)\big(-{\log(1-\alpha)\over\log n}\big)} \tosim (1-\alpha)^x,
$$
which proves that
$$
\frac{\bE f_k(G_{n,d}^\blackdiamond)}{\binom{n+1-2\sigma_\blackdiamond}k} \tosim (1-\alpha)^{x-1}.
$$

The last case $\alpha=1$, that is, $k=n+o(n)$ follows from a monotonicity argument. To prove this, fix some $\eps\in (0,1)$ and let $k'(n):=[k(n)(1-\eps)]$. Then, we obtain $k'(n)=(1-\eps)n+o(n)$, which enables us to apply the previous case for $k'$ with $\alpha=1-\eps\in (0,1)$. Also, since $S_n^\blackdiamond$ has the same distribution as $\sum_{k=1}^{n}\text{Bern}(\sigma_\blackdiamond/k)$ there is a natural coupling of $S_{n-k}^\blackdiamond$ and $S_{n-k'}^\blackdiamond$ such that $S_{n-k}^\blackdiamond\le S_{n-k'}^\blackdiamond$. Thus, we can use~\eqref{eq:face_numbers_2} to obtain
\begin{align*}
\frac{\bE f_k(G_{n,d}^\blackdiamond)}{\binom{n+1-2\sigma_\blackdiamond}k}
	&=\frac{\sum\limits_{\ell=1,3,\dots}\bP[S_{n-k}^\blackdiamond=n-d+\ell]}{\sum\limits_{\ell=1,3,\dots}\bP[S_n^\blackdiamond=n-d+\ell]}\\
	&\le \frac{\bP[S_{n-k}^\blackdiamond\ge n-d]}{\sum\limits_{\ell=1,3,\dots}\bP[S_n^\blackdiamond=n-d+\ell]}
	\le \frac{\bP[S_{n-k'}^\blackdiamond\ge n-d]}{\sum\limits_{\ell=1,3,\dots}\bP[S_n^\blackdiamond=n-d+\ell]}.
\end{align*}
The latter can be simplified to
\begin{align*}
\frac{\bP[S_{n-k'}^\blackdiamond\ge y_n\sigma_\blackdiamond\log(n-k')]}{\sum\limits_{\ell=1,3,\dots}\bP[S_n^\blackdiamond=x_n\sigma_\blackdiamond\log n+\ell]},
\end{align*}
where, as in the previous case, $y_n:=(x_n\log n)/\log(n-k')$ converges to $x>1$, as $n\to\infty$. However, due to~\eqref{eq:asympt_S_n>=} and~\eqref{eq:AsymptoticsGEQwithZ>1}, this is, up to a constant only depending on $x$, asymptotically equivalent to
\begin{align*}
\frac{\sum\limits_{\ell=1,3,\dots}\bP[S_{n-k'}^\blackdiamond=y_n\sigma_\blackdiamond\log(n-k')+\ell]}{\sum\limits_{\ell=1,3,\dots}\bP[S_n^\blackdiamond=x_n\sigma_\blackdiamond\log n+\ell]},
\end{align*}
which converges to $(1-(1-\eps))^x$, as $n\to\infty$, following the previous case with $\alpha=1-\eps\in (0,1)$. Letting $\eps\downarrow 0$ yields the claim.

\vspace*{2mm}
\noindent
\textsc{Case 2.}
Now, we prove the case where $x\in (0,1)$. In this case, the denominator of~\eqref{eq:face_numbers_1} converges to $1/2$ which  follows from \eqref{eq:AsymptoticsGEQwithZ<1}. In the case $k=n-n^{c+o(1)}$, that is, if $k=n-\exp\{c\log n+o(\log n)\}=n-n^{c_n}$ for some sequence $(c_n)_{n\in\NN}$ such that $c_n\to c$, as $n\to\infty$, the sequence $(y_n)_{n\in\NN}$, $y_n:=(x_n\log n)/\log (n-k)$ can be written as $y_n=x_n/c_n$. If $c\in(0,x)$, $y_n$ converges to $x/c>1$ and we can apply~\eqref{eq:AsymptoticsGEQwithZ>1} to the numerator of~\eqref{eq:face_numbers_1} to deduce that
\begin{align*}
\lim_{n\to\infty} \frac{\bE f_k(G_{n,d}^\blackdiamond)}{\binom{n+1-2\sigma_\blackdiamond}k}=\lim_{n\to\infty}\frac{\sum\limits_{\ell=1,3,\dots}\bP[S_{n-k}^\blackdiamond=(x_n/c_n)\sigma_\blackdiamond\log(n-k)+\ell]}{\sum\limits_{\ell=1,3,\dots}\bP[S_n^\blackdiamond=x_n\sigma_\blackdiamond\log n+\ell]}=0.
\end{align*}
If on the other hand $c\in(x,1)$, we have that $x/c\in(0,1)$. Thus, by~\eqref{eq:AsymptoticsGEQwithZ<1}, the numerator of~\eqref{eq:face_numbers_1} also converges to $1/2$ which yields
\begin{align*}
\lim_{n\to\infty}\frac{\bE f_k(G_{n,d}^\blackdiamond)}{\binom{n+1-2\sigma_\blackdiamond}k}=\lim_{n\to\infty}=\frac{\sum\limits_{\ell=1,3,\dots}\bP[S_{n-k}^\blackdiamond=(x_n/c_n)\sigma_\blackdiamond\log(n-k)+\ell]}{\sum\limits_{\ell=1,3,\dots}\bP[S_n^\blackdiamond=x_n\sigma_\blackdiamond\log n+\ell]}=1.
\end{align*}
In the remaining case where $k=n-\exp\{{\sigma_{\blackdiamond}^{-1}(n-d-\alpha\sqrt{x\sigma_\blackdiamond\log n})}\}$ for some $\alpha\in\RR$, we use that $n-d=\sigma_\blackdiamond x_n\log n$ to obtain
\begin{align*}
\sigma_\blackdiamond\log(n-k)+\alpha\sqrt{\sigma_\blackdiamond\log(n-k)}
&	=n-d-\alpha\sqrt{x\sigma_\blackdiamond\log n}+\alpha\sqrt{n-d-\alpha\sqrt{x\sigma_\blackdiamond\log n}}\\
&	=x_n\sigma_\blackdiamond\log n-\alpha\sqrt{x\sigma_\blackdiamond\log n}+\alpha\sqrt{x_n\sigma_\blackdiamond\log n-\alpha\sqrt{x\sigma_\blackdiamond\log n}}\\
&	=x_n\sigma_\blackdiamond\log n+o(\sqrt{\log (n-k)}),
\end{align*}
since
\begin{align*}
&\frac{-\alpha\sqrt{x\sigma_\blackdiamond\log n}+\alpha\sqrt{x_n\sigma_\blackdiamond\log n-\alpha\sqrt{x\sigma_\blackdiamond\log n}}}{\sqrt{\log (n-k)}}\\
&\quad=	-\alpha\sqrt{\frac{x\sigma_\blackdiamond\log n}{x_n\log n-\sigma_\blackdiamond^{-1}\alpha\sqrt{x\sigma_\blackdiamond\log n}}}+\alpha\sqrt{\frac{x_n\sigma_\blackdiamond\log n}{x_n\log n-\sigma_\blackdiamond^{-1}\alpha\sqrt{x\sigma_\blackdiamond\log n}}+o(1)}=o(1).
\end{align*}
Thus, the numerator of~\eqref{eq:face_numbers_1} can be written as
\begin{align*}
&\sum_{\ell=1,3,\dots}\bP[S_{n-k}^\blackdiamond=x_n\sigma_\blackdiamond\log n+\ell]\\
&\quad	=\sum_{\ell=1,3,\dots}\bP\big[S_{n-k}^\blackdiamond=\sigma_\blackdiamond\log(n-k)+\alpha\sqrt{\sigma_\blackdiamond\log(n-k)}+o\big(\sqrt{\log(n-k)}\big)+\ell\big]\ton\frac 12(1-\Phi(\alpha)),
\end{align*}
which follows from the central-limit-type result~\eqref{eq:CLT_stirling_kind_of}. This completes the proof since the denominator of~\eqref{eq:face_numbers_1} converges two $1/2$.
\end{proof}

The next theorem is a kind of large deviation principle for $\bE f_k(G_{n,d}^\blackdiamond)$ and can be considered an analogue of~\cite[Theorem~4.7]{GKT2020_HighDimension1} in the setting of Weyl random cones.

\begin{theorem}\label{theorem:large_dev_f_k_dual_Weyl_A}
Let $\blackdiamond\in\{A,B\}$ and $G_{n,d}^\blackdiamond$ be the dual of the Weyl random cone $W_{n,d}^\blackdiamond$. Consider the regime $d(n)$ and $k(n)$ with
\begin{align}\label{eq:regime_dual_Weyl_A_k-faces}
d=n-\sigma_\blackdiamond x\log n +o(\log n)\quad\text{and}\quad k=n-\exp\{c\log n+o(\log n)\}, \qquad\text{as }n\to\infty,
\end{align}
where $x>0$ and $c\in(0,1)$. Then, it holds that
\begin{align*}
\lim_{n\to\infty}\frac{1}{\log n}\log\frac{\bE f_k(G_{n,d}^\blackdiamond)}{\binom{n+1-2\sigma_\blackdiamond}k}=
\begin{cases}
x\log c-c+1	&:x>1,\\
x-x\log x+x\log c-c	&: x\in(0,1),c\in(0,x),\\
0						&:x\in(0,1),c\in(x,1).
\end{cases}
\end{align*}
\end{theorem}

\begin{proof}
We use the same notation as in the proof of Theorem \ref{theorem:phase_trans_faces_dual_Weyl_A}, in particular we recall that $y_n={x_n\log n\over\log(n-k)}={x_n\over c_n}$ for sequences $(x_n)_{n\ge 0}$ and $(c_n)_{n\ge 0}$ such that $\lim_{n\to\infty} x_n=x>0$, $\lim_{n\to\infty} c_n=c\in(0,1)$. In the given regime this implies that $d=n-\sigma_\blackdiamond x_n\log n$ and $k=n-n^{c_n}$. Using~\eqref{eq:face_numbers_1}, and that $\sigma_\blackdiamond x_n\log n, n^c_n\in\NN$ for each $\NN$ we obtain
\begin{align*}
\frac{\bE f_k(G_{n,d}^\blackdiamond)}{\binom{n+1-2\sigma_\blackdiamond}k}
&	=\frac{\sum\limits_{\ell=1,3,\dots}\bP[S_{n-k}^\blackdiamond=y_n\sigma_\blackdiamond\log(n-k)+\ell]}{\sum\limits_{\ell=1,3,\dots}\bP[S_n^\blackdiamond=x_n\sigma_\blackdiamond\log n+\ell]}\\
&	=\frac{\sum\limits_{\ell=1,3,\dots}\bP[S_{n-k}^\blackdiamond=\sigma_\blackdiamond(x_n/c_n)\log(n-k)+\ell]}{\sum\limits_{\ell=1,3,\dots}\bP[S_n^\blackdiamond=x_n\sigma_\blackdiamond\log n+\ell]}.
\end{align*}
In the case $x>1$, we already have that $c<x$, and thus, the asymptotic equivalence~\eqref{eq:AsymptoticsGEQwithZ>1} implies
\begin{align*}
\frac{\bE f_k(G_{n,d}^\blackdiamond)}{\binom{n+1-2\sigma_\blackdiamond}k}
&	\tosim \frac{(n-k)^{-((x_n/c_n)\log(x_n/c_n)-(x_n/c_n)+1)}}{n^{-(x_n\log x_n-x_n+1)}}\frac{\sqrt{2\pi \frac xc\log(n-k)}}{\sqrt{2\pi x\log n}}\frac{\Psi_\blackdiamond (\log(x/c))(x^{2\sigma_\blackdiamond}-1)}{\Psi_\blackdiamond(\log x)c^{\sigma_\blackdiamond}((x/c)^{2\sigma_\blackdiamond}-1)}\\
&	\tosim C(x,c)\cdot\frac{n^{-(x_n\log(x_n/c_n)-x_n+c_n)}}{n^{-(x_n\log x_n-x_n+1)}}\\
&	\tosim C(x,c)\cdot n^{x_n\log c_n-c_n+1},
\end{align*}
for some constant $C(x,c)$ independent of $n$. Note that this asymptotic relation is again independent of the choice of $\blackdiamond$.
Consequently, we obtain
\begin{align*}
\lim_{n\to\infty}\frac{1}{\log n}\log\frac{\bE f_k(G_{n,d}^\blackdiamond)}{\binom{n+1-2\sigma_\blackdiamond}k}
	=(x\log c-c+1)\lim_{n\to\infty}\frac{\log n}{\log n}=x\log c-c+1.
\end{align*}

In the case $x\in(0,1)$ and $c\in(0,x)$, the sequence $x_n/c_n$ converges to $x/c>1$ and we can use~\eqref{eq:AsymptoticsGEQwithZ>1} and~\eqref{eq:AsymptoticsGEQwithZ<1} to obtain
\begin{align*}
\frac{\bE f_k(G_{n,d}^\blackdiamond)}{\binom{n+1-2\sigma_\blackdiamond}k}
&	\tosim \frac{2(n-k)^{-((x_n/c_n)\log(x_n/c_n)-(x_n/c_n)+1)}}{\sqrt{2\pi \frac xc\log(n-k)}}\Psi_\blackdiamond\Big(\log\Big(\frac xc\Big)\Big)\frac{x}{(\frac{x^2}{c}-c)}\\
&	\tosim C'(x,c)\cdot\frac{2n^{-(x_n\log(x_n/c_n)-x_n+c_n)}}{\sqrt{2\pi x\log n}},
\end{align*}
for some constant $C'(x,c)$ independent of $n$.
Hence, the result is
\begin{align*}
\lim_{n\to\infty}\frac{1}{\log n}\log\frac{\bE f_k(G_{n,d}^\blackdiamond)}{\binom{n+1-2\sigma_\blackdiamond}k}
&	=-\Big(x\log\Big(\frac xc\Big)-x+c\Big)\lim_{d\to\infty}\frac{\log n}{\log n}-2\lim_{d\to\infty}\frac{\log(\log n)}{\log n}\\
&	=x-x\log x+x\log c-c.
\end{align*}
In the case $x\in(0,1)$ and $c\in(0,x)$, Theorem~\ref{theorem:phase_trans_faces_dual_Weyl_A} yields the claim. This completes the proof.
\end{proof}

\section{Limit theorems for the expected conic intrinsic volumes and quermassintegrals}\label{sec:limit_intr_vol_quermass}

This section contains limit theorems for the expected conic intrinsic volumes of the dual Weyl random cones and limit theorems for the expected quermassintegrals of the Weyl random cones. We start by stating and proving two limit theorems for the expected conic intrinsic volumes of the duals $G_{n,d}^\blackdiamond$ of the Weyl random cones $W_{n,d}^\blackdiamond$. The first one is a kind of large deviation principle, similar to Theorem~\ref{theorem:large_dev_f_k_dual_Weyl_A} and~\cite[Theorem~5.10]{GKT2020_HighDimension1}, while the second one is limit theorem of distributional kind for random variables putting mass $\bE \upsilon_k(G_{n,d}^\blackdiamond)$ on each value $k\in\NN_0$, similar to~\cite[Theorem~5.3]{GKT2020_HighDimension1}.

\begin{theorem}\label{theorem:large_dev_v_k_dual_Weyl_A}
Let $\blackdiamond\in\{A,B\}$ and $G_{n,d}^\blackdiamond$ be the dual of the Weyl random cone $W_{n,d}^\blackdiamond$. Consider the regime $d(n)$ and $k(n)$ with
\begin{align}\label{eq:regime_dual_Weyl_A_intr_vol}
d=n-x\sigma_\blackdiamond \log n+o(\log n)\quad\text{and}\quad k=n-y\sigma_\blackdiamond \log n +o(\log n), \qquad\text{as }d\to\infty,
\end{align}
for parameters $x>0$ and $y>x$. Then, it holds that
\begin{align*}
\frac{1}{\log n}\log \bE \upsilon_k(G_{n,d}^\blackdiamond)=
\begin{cases}
y-y\log y -1		&: x\in (0,1),\\
x\log x-y\log y +y-x	&:x>1.
\end{cases}
\end{align*}
\end{theorem}

\begin{proof}
The regime~\eqref{eq:regime_dual_Weyl_A_intr_vol} implies that
\begin{align*}
n-d=x_n\sigma_\blackdiamond\log n\qquad\text{and}\qquad n-k=y_n\sigma_\blackdiamond\log n
\end{align*}
for sequences $(x_n)_{d\ge 0}$ and $(y_n)_{n\ge 0}$ such that $\lim_{n\to\infty}x_n= x>0$, $\lim_{n\to\infty}y_n= y>x$ and $x_n\sigma_\blackdiamond\log n,y_n\sigma_\blackdiamond\log n\in\NN$ for each $n\in\NN$. Thus, we can assume that $k<d$. Following~\eqref{eq:IntVolWeylDual}, we have
\begin{align}\label{eq:v_k_dual_Weyl_A}
\bE \upsilon_k(G_{n,d}^\blackdiamond)
&	=\frac{\blackdiamond(n,n-k)}{D^\blackdiamond(n,d)}
	=\frac{\bP[S^\blackdiamond_n=n-k]}{2\sum\limits_{\ell=1,3,\dots}\bP[S^\blackdiamond_n=n-d+\ell]}
	=\frac{\bP[S^\blackdiamond_n=y_n\log n]}{2\sum\limits_{\ell=1,3\dots}\bP[S^\blackdiamond_n=x_n\sigma_\blackdiamond\log n+\ell]}.
\end{align}
For $x\in(0,1)$, we know that the denominator of~\eqref{eq:v_k_dual_Weyl_A} converges to $1$, due to~\eqref{eq:AsymptoticsGEQwithZ<1}. Applying~\eqref{eq:AsymptoticsGEQwithZ_noSum} to the numerator yields
\begin{align*}
\bE \upsilon_k(G_{n,d}^\blackdiamond)\tosim {n^{-(y_N\log y_N-y_N+1)}\over\sqrt{2\pi y\log n}}\,\Psi_\blackdiamond(\log y).
\end{align*}
Thus, we obtain
\begin{align*}
\lim_{n\to\infty}\frac{1}{\log n}\log\bE\upsilon_k(G_{n,d}^\blackdiamond)
	&=-(y\log y-y+1)\lim_{n\to\infty}\frac{\log n}{\log n}-\lim_{n\to\infty}\frac{\log(\sqrt{\log n})}{\log n}\\
	&=y-y\log y -1,
\end{align*}
which proves the first case. In the case $x>1$ however, we use~\eqref{eq:AsymptoticsGEQwithZ>1} to obtain
\begin{align*}
\sum_{\ell=1,3\dots}\bP[S_n=x_n\log n+\ell]
	\tosim {n^{-(x_n\log x_n-x_n+1)}\over\sqrt{2\pi x\log n}}\,\Psi_\blackdiamond(\log x)\,{x^{\sigma_\blackdiamond}\over x^{2\sigma_\blackdiamond}-1}.
\end{align*}
Inserting this into~\eqref{eq:v_k_dual_Weyl_A}, yields
\begin{align*}
\bE\upsilon_k(G_{n,d}^\blackdiamond)\tosim \frac{\Psi_\blackdiamond(\log y)}{2\Psi_\blackdiamond(\log x)}\frac{x^{2\sigma_\blackdiamond}-1}{x^{\sigma_\blackdiamond}} n^{-(y_n\log y_n-y_n)+(x_n\log x_n-x_n)}.
\end{align*}
Finally, this implies
\begin{align*}
\lim_{n\to\infty}\frac{1}{\log n}\log\bE\upsilon_k(G_{n,d}^\blackdiamond)
&	=-(y\log y-y)+(x\log x-x)\lim_{n\to\infty}\frac{\log n}{\log n}\\
&	=x\log x-y\log y +y-x,
\end{align*}
which completes the proof.
\end{proof}

Next, we turn to the analysis of the conic intrinsic volume random variable $X^\blackdiamond_{n,d}$, which is the random variable on $\{0,1,\ldots,d\}$ with probability mass function given by the expected conic intrinsic volumes of the dual Weyl random cone $G_{n,d}^\blackdiamond$, that is,
$$
\bP[X^\blackdiamond_{n,d}=k]=\bE \upsilon_k(G_{n,d}^\blackdiamond),\qquad k=0,1,\ldots,d.
$$
Note that since $\bE \upsilon_k(G_{n,d}^\blackdiamond)+\bE \upsilon_1(G_{n,d}^\blackdiamond)+\ldots+\bE \upsilon_d(G_{n,d}^\blackdiamond)=1$, by the very definition of the conic intrinsic volumes, this does indeed define a random variable. We also introduce the following notation. If $N(0,1)$ denotes the standard Gaussian distribution we write $N(0,1)\,|\,\{N(0,1)<c\}$, $c\in\RR$, for a random variable having the conditional standard Gaussian distribution given $\{N(0,1)<c\}$. On other words, $N(0,1)\,|\,\{N(0,1)<c\}$ is a random variable taking values in $(-\infty,c)$ and its distribution function is given by $\Phi(t)/\Phi(c)$, $t<c$.

\begin{theorem}\label{thm:0406}
Let $\blackdiamond\in\{A,B\}$ and $G_{n,d}^\blackdiamond$ be the dual of the Weyl random cone $W_{n,d}^\blackdiamond$ and $X^\blackdiamond_{n,d}$ be the associated conic intrinsic volume random variable. Consider the regime where $d=d(n)$ is such that
\begin{align*}
d=n-x\sigma_\blackdiamond \log n+o(\log n),\qquad\text{as }n\to\infty,
\end{align*}
for a parameter $x>0$.  For $x\in (0,1)$, we have the central limit theorem
\begin{align*}
\frac{X^\blackdiamond_{n,d}-(n-\sigma_\blackdiamond \log n)}{\sqrt{\sigma_\blackdiamond\log n}}\todistr N(0,1).
\end{align*}
In the case $x>1$, it holds that
\begin{align*}
d-X^\blackdiamond_{n,d}\todistr Z_{\blackdiamond,x},
\end{align*}
where $Z_{\blackdiamond,x}$ is a random variable with values in $\NN_0$ and distribution given by
\begin{align*}
\bP[Z_{\blackdiamond,x}=0]=\frac{1}{2}\cdot\frac{x^{\sigma_\blackdiamond}-1}{x^{\sigma_\blackdiamond}},\qquad \bP[Z_{\blackdiamond,x}=k]=\frac{1}{2}\big(x^{\sigma_\blackdiamond}+1\big)\cdot\frac{x^{\sigma_\blackdiamond}-1}{x^{\sigma_\blackdiamond}}\cdot\bigg(1-\frac{x^{\sigma_\blackdiamond}-1}{x^{\sigma_\blackdiamond}}\bigg)^k,\quad k\in\NN.
\end{align*}
In the case where
$$
d=n-\sigma_\blackdiamond \log n+c\sqrt{\sigma_\blackdiamond\log n}+o(\sqrt{\log n}),\qquad\text{as }d\to\infty
$$
for a parameter $c\in\RR$, we obtain
\begin{align*}
\frac{X^\blackdiamond_{n,d}-(n-\sigma_\blackdiamond\log n)}{\sqrt{\sigma_\blackdiamond\log n}}\todistr N(0,1)\,|\,\{N(0,1)<c\}.
\end{align*}
\end{theorem}

\begin{remark}
We recall that a random variable $X$ has a fractional linear distribution with parameters $a,b,c,d\in\RR$ provided that its generating function is a fractional linear function, that is $\EE[s^X]={as+b\over cs+d}$. Distributions of this type often appear in the context of branching processes, see \cite{Harris}. One can verify that the random variable $Z_{\blackdiamond,x}$ in the previous theorem has generating function $\EE[s^{Z_{\blackdiamond,x}}]={(s+1)(x^{\sigma_\blackdiamond}-1)\over 2(x^{\sigma_\blackdiamond}-s)}$ for $|s|<|x^{\sigma_\blackdiamond}|$, and is thus fractional linear with $a=b=x^{\sigma_\blackdiamond}-1$, $c=-2$ and $d=2x^{\sigma_\blackdiamond}$.
\end{remark}

\begin{proof}[Proof of Theorem \ref{thm:0406}]
Start with the case $x\in(0,1)$. Then, for $n$ sufficiently large and all $t\in\RR$, we have that
\begin{align*}
n-\sigma_\blackdiamond\log n+t\sqrt{\sigma_\blackdiamond\log n}=d+(x-1)\sigma_\blackdiamond \log d+o(\log d)<d.
\end{align*}
Thus, it follows from~\eqref{eq:IntVolWeylDual} and~\eqref{eq:v_k_dual_Weyl_A} that
\begin{align*}
\bP\left[X^\blackdiamond_{n,d}\le n-\sigma_\blackdiamond\log n+t\sqrt{\sigma_\blackdiamond\log n}\right]
&	=\sum_{k=0}^{\lfloor n-\sigma_\blackdiamond\log n+t\sqrt{\sigma_\blackdiamond\log n}\rfloor}\bE \upsilon_k(G_{n,d}^\blackdiamond)\\
&	=\sum_{k=0}^{\lfloor n-\sigma_\blackdiamond\log n+t\sqrt{\sigma_\blackdiamond\log n}\rfloor}\frac{\bP[S^\blackdiamond_n=n-k]}{2\sum\limits_{\ell=1,3\dots}\bP[S^\blackdiamond_n=n-d+\ell]}\\
&	=\sum_{k=0}^{\lfloor n-\sigma_\blackdiamond\log n+t\sqrt{\sigma_\blackdiamond\log n}\rfloor}\frac{\bP[S^\blackdiamond_n=n-k]}{2\sum\limits_{\ell=1,3\dots}\bP[S^\blackdiamond_n=x_n\sigma_\blackdiamond\log n+\ell]},
\end{align*}
where we used that $n-d=x_n\sigma_\blackdiamond\log n$ for a sequence $(x_n)_{n\ge 0}$ with $\lim_{n\to\infty} x_n=x$ and $x_n\sigma_\blackdiamond\log n\in\NN$ for each $n\in\NN$.
By~\eqref{eq:AsymptoticsGEQwithZ<1}, the denominator converges to $1$ as $d\to\infty$, which implies
\begin{align*}
\bP\left[X^\blackdiamond_{n,d}\le n-\sigma_\blackdiamond\log n+t\sqrt{\sigma_\blackdiamond\log n}\right]
&	\tosim \sum_{k=0}^{\lfloor n-\sigma_\blackdiamond\log n+t\sqrt{\sigma_\blackdiamond\log n}\rfloor}\bP[S^\blackdiamond_n=n-k]\\
&	\hspace*{2.2mm}=\bP\left[S^\blackdiamond_n\ge \sigma_\blackdiamond\log n-t\sqrt{\sigma_\blackdiamond\log n}\right]\ton 1-\Phi(-t)=\Phi(t),
\end{align*}
where we used the central limit theorem~\eqref{eq:CLT_Stirling} for the random variables $S_n^\blackdiamond$. This proves the first claim.

In the case $x>1$, we obtain for $k\in\NN$ that
\begin{align*}
\bP[d-X^\blackdiamond_{n,d}=k]
	=\bP[X^\blackdiamond_{n,d}=d-k]	
	=\frac{\blackdiamond(n,n-d+k)}{D^\blackdiamond(n,d)}
	=\frac{\bP[S^\blackdiamond_n=x_n\sigma_\blackdiamond\log n+k]}{2\sum\limits_{\ell=1,3,\dots}\bP[S^\blackdiamond_n=x_n\sigma_\blackdiamond\log n+\ell]}.
\end{align*}
The asymptotic equivalences~\eqref{eq:AsymptoticsGEQwithZ_noSum} and~\eqref{eq:AsymptoticsGEQwithZ>1} applied to the numerator and denominator, respectively, yields
\begin{align*}
\bP[d-X^\blackdiamond_{n,d}=k]
	\ton \frac{x^{-\sigma_\blackdiamond k}(x^{2\sigma_\blackdiamond}-1)}{2x^{\sigma_\blackdiamond}}=\frac{1}{2}(x^{\sigma_\blackdiamond}+1)\cdot\frac{x^{\sigma_\blackdiamond}-1}{x^{\sigma_\blackdiamond}}\cdot\bigg(1-\frac{x^{\sigma_\blackdiamond}-1}{x^{\sigma_\blackdiamond}}\bigg)^k.
\end{align*}
For $k=0$, we use~\eqref{eq:IntVolWeylDual} to deduce that
\begin{align*}
\bP[d-X^\blackdiamond_{n,d}=0]=\bP[X^\blackdiamond_{n,d}=d]=\frac{1}{2}-\frac 12\frac{D^\blackdiamond(n,d-1)}{D^\blackdiamond(n,d)}=\frac 12-\frac{\sum\limits_{\ell=1,3,\dots}\bP[S^\blackdiamond_n=x_n\sigma_\blackdiamond\log n+\ell+1]}{2\sum\limits_{\ell=1,3,\dots}\bP[S^\blackdiamond_n=x_n\sigma_\blackdiamond\log n+\ell]}.
\end{align*}
We can apply~\eqref{eq:AsymptoticsGEQwithZ>1} directly in the denominator, while in the numerator we have to apply the analogous result with $\ell$ replaced by $\ell+1$ (which follows from~\eqref{eq:AsymptoticsGEQwithZ_noSum} combined with the dominated convergence theorem). This yields
\begin{align*}
\bP[d-X^\blackdiamond_{n,d}=0]\ton \frac 12-\frac{\sum\limits_{\ell=0}^\infty x^{-\sigma_\blackdiamond(\ell+1)}}{2\sum\limits_{\ell=1,3,\dots} x^{-\sigma_\blackdiamond\ell}}=\frac{1}{2}\cdot\frac{x^{\sigma_\blackdiamond}-1}{x^{\sigma_\blackdiamond}},
\end{align*}
which completes the proof of the second claim.

In the regime $d=n-\sigma_\blackdiamond\log n+c\sqrt{\sigma_\blackdiamond\log n}+o(\sqrt{\log n})$, relation~\eqref{eq:stirling_numbers_rest} combined with the central-limit-type result~\eqref{eq:CLT_stirling_kind_of} yields
\begin{align*}
\frac{\sigma_\blackdiamond^n D^\blackdiamond(n,d)}{n!}
&=	1-2\sum_{\ell=0}^\infty\bP[S_n^\blackdiamond=n-d-2\ell-1]\\
&=	1-2\sum_{\ell=0}^\infty\bP[S_n^\blackdiamond=\sigma_\blackdiamond\log n-c_n\sqrt{\sigma_\blackdiamond\log n}-2\ell-1]\ton \Phi(c).
\end{align*}
Furthermore, for $t<c$ and sufficiently large $d$, we have
\begin{align*}
n-\sigma_\blackdiamond\log n+t\sqrt{\sigma_\blackdiamond\log n}
&	=d-(c-t)\sqrt{\sigma_\blackdiamond\log d}+o(\sqrt{\log d})<d.
\end{align*}
Thus, we obtain
\begin{align*}
\bP\left[X_{n,d}\le n-\sigma_\blackdiamond\log n+t\sqrt{\sigma_\blackdiamond\log n}\right]
&	=\sum_{k=0}^{\lfloor n-\sigma_\blackdiamond\log n+t\sqrt{\sigma_\blackdiamond\log n}\rfloor}\frac{\blackdiamond(n,n-k)}{D^\blackdiamond(n,d)}\\
&	=\frac{n!}{\sigma_\blackdiamond^n D^\blackdiamond(n,d)}\sum_{k=0}^{\lfloor n-\sigma_\blackdiamond\log n+t\sqrt{\sigma_\blackdiamond\log n}\rfloor}\bP[S^\blackdiamond_n=n-k]\\
&	\hspace*{-2.2mm}\tosim \frac{1}{\Phi(c)}\bP[S^\blackdiamond_n\ge \sigma_\blackdiamond\log n-t\sqrt{\sigma_\blackdiamond\log n}]\ton\frac{\Phi(t)}{\Phi(c)}.
\end{align*}
For $t> c$, and sufficiently large $d$, we observe that $n-\sigma_\blackdiamond\log n+t\sqrt{\sigma\blackdiamond\log n}> d$, which yields the claim.
\end{proof}

Next, we collect limit theorems for the expected conic quermassintegrals of the Weyl random cones $W_{n,d}^\blackdiamond$. The first theorem can be considered as analogue to~\cite[Theorem 4]{HugSchneiderThresholdPhenomena} and uncovers a phase transition for $\bE U_k(W_{n,d}^\blackdiamond)$ for each fixed $k\in\NN_0$. The second theorem considers the same quantity but in a regime where $k$ also tends to infinity, similar to~\cite[Theorem~9]{HugSchneiderThresholdPhenomena}.

\begin{theorem}\label{theorem:phase_trans_U_k_Weyl A}
Let $\blackdiamond\in\{A,B\}$ and let $W_{n,d}^\blackdiamond$ be the Weyl random cone. Consider the regime $d=d(n)$ such that
$$
d=n-x\sigma_\blackdiamond\log n+o(\log n),\qquad\text{as }n\to \infty,
$$ where $x>0$. Then, it holds that
\begin{align*}
\lim_{d\to\infty} 2\bE U_k(W_{n,d}^\blackdiamond)=
\begin{cases}
1				&: x\in(0,1),\\
x^{-\sigma_\blackdiamond k}			&: x>1,
\end{cases}
\end{align*}
for any fixed $k\in\NN$.
\end{theorem}

\begin{proof}
In the given regime we have $d=n-x_n\sigma_\blackdiamond\log n$ for a sequence $(x_n)_{n\ge 0}$ such that $\lim_{n\to\infty}x_n= x$ and $x_n\sigma_\blackdiamond\log n\in\NN$ for each $n\in\NN$. Using~\eqref{eq:QuermassIntWeyl}, we obtain
\begin{align}\label{eq:quermass_weyl_probab}
2\bE U_k(W_{n,d}^\blackdiamond)
&	=\frac{D^\blackdiamond(n,d-k)}{D^\blackdiamond(n,d)}\notag\\
&	=\frac{\sum\limits_{\ell=1,3,\dots}\bP[S^\blackdiamond_n=n-d+k+\ell]}{\sum\limits_{\ell=1,3,\dots}\bP[S^\blackdiamond_n=n-d+\ell]}\notag\\
&	=\frac{\sum\limits_{\ell=1,3,\dots}\bP[S^\blackdiamond_n=x_n\sigma_\blackdiamond\log n+k+\ell]}{\sum\limits_{\ell=1,3,\dots}\bP[S^\blackdiamond_n=x_n\sigma_\blackdiamond\log n+\ell]}.
\end{align}
For $x>1$ we can insert the asymptotic equivalence~\eqref{eq:AsymptoticsGEQwithZ>1} in both the numerator and the denominator to obtain
\begin{align*}
\lim_{n\to\infty}2\bE U_k(W_{n,d}^\blackdiamond)
	= x^{-\sigma_\blackdiamond k}.
\end{align*}
For  $x\in(0,1)$ we apply~\eqref{eq:AsymptoticsGEQwithZ<1}, which yields the claim.
\end{proof}

We turn now to the case, where $k$ tends to infinity as well.

\begin{theorem}
Let $W_{n,d}^\blackdiamond$ be a Weyl random cone of type $\blackdiamond\in\{A,B\}$. Consider the regime $d=d(n)$ such that
\begin{align*}
d=n-x\sigma_\blackdiamond\log n+o(\log n)\quad \text{and}\quad k=y\sigma_\blackdiamond\log n+o(\log n),\qquad\text{as }n\to\infty,
\end{align*}
for parameters $x,y>0$. Then, it holds that
\begin{align*}
\lim_{n\to\infty}2\bE U_k(W_{n,d}^\blackdiamond)=
\begin{cases}
1			&:y<\max\{0,1-x\},\\
0			&:y>\max\{0,1-x\}.
\end{cases}
\end{align*}
\end{theorem}

\begin{proof}
The given regime implies that we have
\begin{align*}
n-d=x_n\sigma_\blackdiamond\log n\qquad\text{and}\qquad k=y_n\sigma_\blackdiamond\log n,
\end{align*}
for some sequences $(x_n)_{n\ge 0}$ and $(y_n)_{n\ge 0}$ such that $\lim_{n\to\infty}x_n= x>0$, $\lim_{n\to\infty}y_n= y>0$ and $x_n\sigma_\blackdiamond\log n,y_n\sigma_\blackdiamond\log n\in\NN$ for each $n\in\NN$. Inserting this into~\eqref{eq:quermass_weyl_probab} yields
\begin{align}\label{eq:quermass_2}
2\bE U_k(W_{n,d}^\blackdiamond)
&	=\frac{\sum\limits_{\ell=1,3,\dots}\bP[S^\blackdiamond_n=(x_n+y_n)\sigma_\blackdiamond \log n+\ell]}
{\sum\limits_{\ell=1,3,\dots}\bP[S^\blackdiamond_n=x_n \sigma_\blackdiamond \log n+\ell]}.
\end{align}

Now, suppose that $y<\max\{0,1-x\}$. This already implies that $x\in(0,1)$ and $y+x\in(0,1)$ since in the case where $x\ge 1$ we would obtain $y<0$ which is impossible by assumption. By~\eqref{eq:AsymptoticsGEQwithZ<1} both the numerator and denominator of~\eqref{eq:quermass_2} converge to $1/2$ which yields
\begin{align*}
\lim_{n\to\infty}2\bE U_k(W_{n,d}^\blackdiamond)=1.
\end{align*}

The case $y>\max\{0,1-x\}$ has to be divided into three separate cases. First, suppose $x\in (0,1)$ while $x+y>1$. Then, by~\eqref{eq:AsymptoticsGEQwithZ<1}, the denominator of~\eqref{eq:quermass_2} converges to $1/2$ while the numerator converges to $0$, by \eqref{eq:AsymptoticsGEQwithZ>1}. This yields $\lim_{n\to\infty}2\bE U_k(W_{n,d}^\blackdiamond)=0$. On the other hand, if $x>1$ and $y>0$, we can apply~\eqref{eq:AsymptoticsGEQwithZ>1} in both the numerator and the denominator of~\eqref{eq:quermass_2} to obtain
\begin{align*}
2\bE U_k(W_{n,d}^\blackdiamond)\tosim C(x,y)\cdot \frac{n^{-((x_n+y_n)\log(x_n+y_n)-(x_n+y_n)+1)}}{n^{-(x_n\log x_n-x_n+1)}},
\end{align*}
where $C(x,y)$ is some constant which depends on $x$ and $y$. But since $\cI(a):=a\log a -a+1$ is strictly increasing for $a>0$ (its derivative is given by $\log a$) and $x+y>y$, we observe that
\begin{align*}
2\bE U_k(W_{n,d}^\blackdiamond)\tosim C(x,y) \cdot n^{-(\cI(x_n+y_n)-\cI(x_n))}\ton 0.
\end{align*}
It remains to show that for $x=1$ (and arbitrary $y>0$) the expectation $2\bE U_k(W_{n,d}^\blackdiamond)$ also converges to $0$, as $n\to\infty$. To this end, let $(z_n)_{n\ge 0}$ be a sequence such that $z_n\to y/2$, as $n\to\infty$, and $z_n\sigma_\blackdiamond\log n$ is a positive and even integer for each $n\in\NN$. Then, we obtain
\begin{align*}
\sum\limits_{\ell=1,3,\dots}\bP[S^\blackdiamond_n=x_n \sigma_\blackdiamond \log n+\ell]\ge \sum\limits_{\ell=1,3,\dots}\bP[S^\blackdiamond_n=(x_n+z_n) \sigma_\blackdiamond \log n+\ell].
\end{align*}
Inserting this into~\eqref{eq:quermass_2} yields
\begin{align*}
2\bE U_k(W_{n,d}^\blackdiamond)
	\le \frac{\sum\limits_{\ell=1,3,\dots}\bP[S^\blackdiamond_n=((x_n+z_n)+(y_n-z_n))\sigma_\blackdiamond \log n+\ell]}
{\sum\limits_{\ell=1,3,\dots}\bP[S^\blackdiamond_n=(x_n+z_n) \sigma_\blackdiamond \log n+\ell]}\ton 0,
\end{align*}
which follows from the previous case with $x_n$ replaced by $x_n+z_n$ and $y_n$ replaced by $y_n-z_n$, and thus, $x$ replaced by $x+y/2$ and $y$ replaced by $y/2$.
\end{proof}

\section{Limit theorems for expected statistical dimension}\label{sec:limit_stat_dim}

The \textit{statistical dimension} $\Delta(C)$ of a cone $C\subset\RR^d$ is defined as
\begin{align*}
\Delta(C):=\sum_{j=0}^dj\upsilon_j(C).
\end{align*}
The statistical dimension can be viewed as the conical extension of the dimension of a subspace, see~\cite[Section~5.3]{ALMT14}. In particular, if $L\subset\RR^d$ is an $\ell$-dimensional linear subspace for some $\ell\in\{0,1,\ldots,d\}$ then $\upsilon_k(L)=1$ in the case where $k=\ell$ and $0$ otherwise, which yields $\Delta(L)=\ell\cdot\upsilon_\ell(L)=\ell$. For more properties and an extensive account on the statistical dimension, we refer to~\cite{ALMT14}.

In this section, our goal is to understand the asymptotic behaviour of $\bE\Delta(W_{n,d}^\blackdiamond)$, for $\blackdiamond\in\{A,B\}$, as $n\to\infty$ and $d=d(n)\to\infty$ simultaneously. Following~\eqref{eq:IntVolWeyl}, the expected statistical dimension of $W_{n,d}^\blackdiamond$ is given by
\begin{align}\label{eq:stat_dim_Weyl_A}
\bE\Delta(W_{n,d}^\blackdiamond)=\sum_{k=0}^dk\frac{\blackdiamond(n,n-d+k)}{D^\blackdiamond(n,d)}=\frac{\sum_{\ell=0}^d(d-\ell)\cdot\blackdiamond(n,n-\ell)}{D^\blackdiamond(n,d)}.
\end{align}
Our next result is the analogue of \cite[Theorem 6.3]{GKT2020_HighDimension1} for Weyl cones.

\begin{theorem}\label{theorem:asym_Weyl_cone}
Let $\blackdiamond\in\{A,B\}$ and let $W_{n,d}^\blackdiamond$ be the Weyl random cone. Consider the regime where $d=d(n)$ such that
$$
d=n-x\sigma_\blackdiamond\log n+o(\log n),\qquad\text{as }n\to\infty.
$$
Then, it holds that
\begin{align*}
\bE\Delta(W_{n,d}^\blackdiamond)\tosim
\begin{cases}
\sigma_\blackdiamond\log n 		&: x\in[0,1),\\
\frac{x^{\sigma_\blackdiamond}+1}{2(x^{\sigma_\blackdiamond}-1)}	&: x>1.	
\end{cases}
\end{align*}
In the critical case where
$$
d=n-\sigma_\blackdiamond\log n+c\sqrt{\sigma_\blackdiamond\log n}+o(\sqrt{\log n}),\qquad\text{as }n\to\infty
$$
for a parameter $c\in\RR$,  it holds that
\begin{align*}
\bE\Delta(W_{n,d}^\blackdiamond)\tosim \sqrt{\sigma_\blackdiamond\log d}\left(\frac{e^{-c^2/2}}{\sqrt{2\pi}\Phi(-c)}-c\right).
\end{align*}
\end{theorem}

\begin{proof}
In the regime $n=d+x\sigma_{\blackdiamond}\log d+o(\log d)$, we can find a sequence $(x_n)_{n\ge 0}$ such that $\lim_{n\to\infty}x_n= x>0$, $x_n\log n\in\NN$ for each $n\in\NN$, and $n-d=x_n\sigma_\blackdiamond\log n$.
We start by proving the first case where $x\in [0,1)$. In view of~\eqref{eq:stat_dim_Weyl_A}, we first want to determine the asymptotic behaviour of $D^\blackdiamond(n,d)$. Using~\eqref{eq:stirling_numbers_rest}, we obtain
\begin{align*}
1\ge \frac{\sigma_\blackdiamond^nD^\blackdiamond(n,d)}{n!}
&	=\frac{2\sigma_\blackdiamond^n}{n!}\left(\blackdiamond(n,n-d+1)+\blackdiamond(n,n-d+3)+\ldots\right)\notag\\
&	=1-\frac{2\sigma_\blackdiamond^n}{n!}\left(\blackdiamond(n,n-d-1)+\blackdiamond(n,n-d-3)+\ldots\right)\\
&	\ge 1-2\bP[S_n^\blackdiamond\le x_n\log n]\ton 1,
\end{align*}
due to the weak law of large numbers~\eqref{eq:weakLLN}, for $x\in[0,1)$. Thus, $\sigma_\blackdiamond^nD^\blackdiamond(n,d)/n!$ converges to $1$, as $n\to\infty$, and, for the case $x\in[0,1)$, it remains to prove that
\begin{align*}
\sum_{k=0}^dk\frac{\blackdiamond(n,n-d+k)\sigma_\blackdiamond^n}{n!}\tosim \sigma_\blackdiamond\log n.
\end{align*}
To this end, we split up the sum and obtain, for fixed $\varepsilon>0$,
\begin{align}
&\sum_{k=0}^d\frac{k}{\sigma_\blackdiamond\log n}\frac{\blackdiamond(n,n-d+k)\sigma_\blackdiamond^n}{n!}\notag\\
&	\quad=\sum_{l=x_n\sigma_\blackdiamond\log n}^{d}\frac{\ell}{\sigma_\blackdiamond\log n}\frac{\blackdiamond(n,l)\sigma_\blackdiamond^n}{n!}\notag\\
&	\quad=\sum_{\ell\in[x_n\sigma_\blackdiamond\log n,(1-\eps)\sigma_\blackdiamond\log n)}\frac{\ell}{\sigma_\blackdiamond\log n}\frac{\blackdiamond(n,\ell)\sigma_\blackdiamond^n}{n!}+\sum_{\ell\in[(1-\eps)\sigma_\blackdiamond\log n,(1+\eps)\sigma_\blackdiamond\log n]}\frac{\ell}{\sigma_\blackdiamond\log n}\frac{\blackdiamond(n,\ell)\sigma_\blackdiamond^n}{n!}\label{eq:2sums}\\
&	\quad\quad+\sum_{\ell\in((1+\eps)\sigma_\blackdiamond\log n,d]}\frac{\ell}{\sigma_\blackdiamond\log n}\frac{\blackdiamond(n,\ell)\sigma_\blackdiamond^n}{n!},\label{eq:1sum}
\end{align}
where each sum runs through all integer values in the given interval. The second sum of~\eqref{eq:2sums} is easily treated using the law of large numbers~\eqref{eq:weakLLN}:
\begin{align*}
\sum_{\ell\in[(1-\eps)\sigma_\blackdiamond\log n,(1+\eps)\sigma_\blackdiamond\log n]}\frac{\ell}{\log n}\frac{\blackdiamond(n,\ell)\sigma_\blackdiamond^n}{n!}
&	\le (1+\eps)\sum_{\ell\in[(1-\eps)\sigma_\blackdiamond\log n,(1+\eps)\sigma_\blackdiamond\log n]}\frac{\blackdiamond(n,\ell)\sigma_\blackdiamond^n}{n!}\\
&	=(1+\eps)\bP\left[S^\blackdiamond_n/\sigma_\blackdiamond\log n\in [1-\eps,1+\eps]\right]\ton 1,
\end{align*}
which yields
\begin{align*}
\limsup_{n\to\infty}\sum_{\ell\in[(1-\eps)\sigma_\blackdiamond\log n,(1+\eps)\sigma_\blackdiamond\log n]}\frac{\ell}{\log n}\frac{\blackdiamond(n,\ell)\sigma_\blackdiamond^n}{n!}\le (1+\eps).
\end{align*}
Similarly, we obtain
\begin{align*}
\liminf_{n\to\infty}\sum_{\ell\in[(1-\eps)\sigma_\blackdiamond\log n,(1+\eps)\sigma_\blackdiamond\log n]}\frac{\ell}{\log n}\frac{\blackdiamond(n,\ell)\sigma_\blackdiamond^n}{n!}\ge (1-\eps).
\end{align*}
For the first sum of~\eqref{eq:2sums}, we obtain
\begin{align*}
\sum_{\ell\in[x_n\sigma_\blackdiamond\log n,(1-\eps)\sigma_\blackdiamond\log n)}\frac{\ell}{\sigma_\blackdiamond\log n}\frac{\blackdiamond(n,\ell)\sigma_\blackdiamond^n}{n!}\le \sum_{\ell<(1-\eps)\sigma_\blackdiamond\log n}\frac{\blackdiamond(n,\ell)\sigma_\blackdiamond^n}{n!}=\bP[S_n^\blackdiamond/\sigma_\blackdiamond\log n\le 1-\eps]\ton 0,
\end{align*}
again using the law of large numbers~\eqref{eq:weakLLN}. The sum in~\eqref{eq:1sum} requires slightly more effort. We further split up the sum to obtain
\begin{align*}
&\sum_{\ell\in((1+\eps)\sigma_\blackdiamond\log n,d]}\frac{\ell}{\sigma_\blackdiamond\log n}\frac{\blackdiamond(n,\ell)\sigma_\blackdiamond^n}{n!}\\
&	\quad=\sum_{\ell\in((1+\eps)\sigma_\blackdiamond\log n,10\,\sigma_\blackdiamond\log n]}\frac{\ell}{\sigma_\blackdiamond\log n}\frac{\blackdiamond(n,\ell)\sigma^n_\blackdiamond}{n!}+\sum_{\ell\in(10\,\sigma_\blackdiamond\log n,d]}\frac{\ell}{\sigma_\blackdiamond\log n}\frac{\blackdiamond(n,\ell)\sigma^n_\blackdiamond}{n!}\\
&	\quad\le 10 \sum_{\ell\in((1+\eps)\sigma_\blackdiamond\log n,10\,\sigma_\blackdiamond\log n)}\frac{\blackdiamond(n,\ell)\sigma_\blackdiamond^n}{n!}+ n\sum_{\ell=10\,\sigma_\blackdiamond\log n}^\infty\frac{\blackdiamond(n,\ell)\sigma_\blackdiamond^n}{n!}.
\end{align*}
Note that we used $d\le n\sigma_\blackdiamond\log n$ in the last step.
The first sum converges to $0$ due to the law of large numbers, while
for the second sum we use~\eqref{eq:asympt_S_n>=} to obtain
\begin{align*}
n\sum_{\ell=10\,\sigma_\blackdiamond\log n}^\infty\frac{\ell}{\sigma_\blackdiamond\log n}\frac{\blackdiamond(n,\ell)\sigma_\blackdiamond^n}{n!}=n\bP[S_n^\blackdiamond\ge 10\sigma_\blackdiamond\log n]\tosim  n\frac{n^{-(10\log 10 -10+1)}}{\sqrt{2\pi 10\log n}}\frac{10}{9}\Psi_\blackdiamond(\log 10)\ton 0.
\end{align*}
Note that to be formally correct we would need to replace $10$ by a sequence $(z_n)_{n\ge 0}$ that converges to $10$ and satisfies $z_n\sigma_\blackdiamond\log n\in\NN$ for all $n\in\NN$.
Altogether, we obtain
\begin{align*}
\limsup_{n\to\infty}\sum_{k=0}^d\frac{k}{\sigma_\blackdiamond\log n}\frac{\blackdiamond(n,n-d+k)\sigma_\blackdiamond^n}{n!}\le (1+\eps), \quad\liminf_{n\to\infty}\sum_{k=0}^d\frac{k}{\sigma_\blackdiamond\log n}\frac{\blackdiamond(n,n-d+k)\sigma_\blackdiamond^n}{n!}\ge (1-\eps).
\end{align*}
Letting $\eps\downarrow 0$ yields
\begin{align*}
\sum_{k=0}^dk\frac{\blackdiamond(n,n-d+k)\sigma_n^\blackdiamond}{n!}\tosim \sigma_\blackdiamond\log n,
\end{align*}
which completes the proof of the case $x\in[0,1)$.

Now, we turn to the case where $x>1$ (and still $d=n-x_n\sigma_\blackdiamond\log n$). In this case, we have
\begin{align*}
\frac{\sigma_\blackdiamond^nD^\blackdiamond(n,d)}{n!}=2\sum_{l=1,3,\dots}\bP[S_n^\blackdiamond=x_n\sigma_\blackdiamond\log n+l]\tosim{N^{-(x_n\log x_n-x_n+1)}\over\sqrt{2\pi x\log n}}\,\Psi_\blackdiamond(\log x)\,{2x^{\sigma_\blackdiamond}\over x^{2\sigma_\blackdiamond}-1},
\end{align*}
by~\eqref{eq:AsymptoticsGEQwithZ>1}.
Furthermore, it holds that
\begin{align*}
\sum_{k=0}^dk\frac{\blackdiamond(n,n-d+k)\sigma_\blackdiamond^n}{n!}
&	=\sum_{k=0}^dk\frac{\blackdiamond(n,x_n\sigma_\blackdiamond\log n+k)\sigma_\blackdiamond^n}{n!}\\
&	\hspace*{-2.2mm}\tosim\sum_{k=0}^\infty{n^{-(x_n\log x_n-x_n+1)}\over\sqrt{2\pi x\log n}}\,\Psi_\blackdiamond(\log x)\,kx^{-\sigma_\blackdiamond k}\\
&	=\frac{n^{-(x_n\log n -x_n+1)}}{\sqrt{2\pi x\log n}}\Psi_\blackdiamond(\log x)\frac{x^{\sigma_\blackdiamond}}{(x^{\sigma_\blackdiamond}-1)^2},
\end{align*}
using~\eqref{eq:AsymptoticsGEQwithZ_noSum} combined with the dominated convergence theorem. Hence, we have
\begin{align*}
\bE\Delta(W_{n,d}^\blackdiamond)=\frac{\sum_{k=0}^dk\frac{\blackdiamond(n,n-d+k)}{n!}}{\sigma_\blackdiamond^nD^\blackdiamond(n,d)/n!}\tosim \frac{x^{2\sigma_\blackdiamond}-1}{2(x^{\sigma_\blackdiamond}-1)^2}=\frac{x^{\sigma_\blackdiamond}+1}{2(x^{\sigma_\blackdiamond}-1)},
\end{align*}
completing the proof of this case.
 
We finally turn to the critical regime where $d=n-\sigma_\blackdiamond\log n+c\sqrt{\sigma_\blackdiamond\log n}+o(\sqrt{\log n})$, as $d\to\infty$. Equivalently, we can write
$
d=n-\sigma_\blackdiamond\log n-c_n\sqrt{\sigma_\blackdiamond\log n}
$
for a sequence $(c_n)_{n\ge 0}$ such that $\lim_{n\to\infty}c_n=c\in \RR$. We can apply~\eqref{eq:stirling_numbers_rest} and the central-limit-type result~\ref{eq:CLT_stirling_kind_of} to obtain
\begin{align}\label{eq:asym_D^A(N,n)_critical}
\frac{\sigma_\blackdiamond^nD^\blackdiamond(n,d)}{n!}&=\frac{2\sigma_\blackdiamond^n}{n!}\sum_{\ell=1,3,\dots}\blackdiamond(n,n-d+l)=1-\frac{2\sigma_\blackdiamond^n}{n!}\sum_{\ell=1,3,\dots}\blackdiamond(n,n-d-l)\notag1\\
&	=1-2\sum_{\ell=1,3,\ldots}\bP\big[S_n^\blackdiamond = \sigma_\blackdiamond\log n+c_n\sqrt{\sigma_\blackdiamond\log n}-\ell\big]\ton 1-\Phi(c)=\Phi(-c),
\end{align}
which implies that $D^\blackdiamond(n,d)\sim \Phi(-c)n!/\sigma_\blackdiamond^n$ as $n\to\infty$. In view of~\eqref{eq:stat_dim_Weyl_A}, it is left to consider the sum $\sum_{\ell=0}^d(d-\ell)\cdot\blackdiamond(n,n-\ell)$ which can be rewritten as follows:
\begin{align}\label{eq:asym_Weyl_sum}
\sum_{\ell=0}^d(d-\ell)\cdot\blackdiamond(n,n-\ell)
&	=\sum_{\ell=0}^d(n-\ell)\cdot\blackdiamond(n,n-\ell)-\sum_{\ell=0}^d(n-d)\cdot\blackdiamond(n,n-\ell)\notag\\
&	=\sum_{\ell=n-d}^n\ell\cdot\blackdiamond(n,\ell)-(n-d)\sum_{\ell=n-d}^n\blackdiamond(n,l)\notag\\
&	=\frac{n!}{\sigma_\blackdiamond^n}\cdot\left(\bE\big[S_n^\blackdiamond\ind_{\{n-d\le S^\blackdiamond_n\le n\}}\big]-(n-d)\bP[n-d\le S_n^\blackdiamond\le n]\right).
\end{align}
Defining $Z_n^\blackdiamond:=(S_n^\blackdiamond-\sigma_\blackdiamond\log n)/\sqrt{\sigma_\blackdiamond\log n}$, we obtain
\begin{align}\label{eq:asym_1}
\bP[n-d\le S_n^\blackdiamond\le n]
&	=\bP\left[c+o(1)\le Z_n^\blackdiamond\le \frac{n-\sigma_\blackdiamond\log n}{\sqrt{\sigma_\blackdiamond\log n}}\right]\ton 1-\Phi(c)=\Phi(-c),
\end{align}
following the central limit theorem~\eqref{eq:CLT_Stirling}.
Similarly,
\begin{align*}
\bE\big[S_n^\blackdiamond\ind_{\{n-d\le Z_n^\blackdiamond\le n\}}\big]\notag	&=\bE\big[\big(\sqrt{\sigma_\blackdiamond\log n}\cdot Z_n^\blackdiamond+\sigma_\blackdiamond\log n)\ind_{\{c+o(1)\le Z_n^\blackdiamond\le (n-\sigma_\blackdiamond\log n)/\sqrt{\sigma_\blackdiamond\log n}\}}\big]\notag\\
&=\sqrt{\sigma_\blackdiamond\log n}\cdot\bE\big[Z_n^\blackdiamond\ind_{\{c+o(1)\le Z_n^\blackdiamond\le (n-\sigma_\blackdiamond\log n)/\sqrt{\sigma_\blackdiamond\log n}\}}\big]\notag\\
&\qquad\qquad+\sigma_\blackdiamond\log n\bP\left[c+o(1)\le Z_n^\blackdiamond\le \frac{n-\sigma_\blackdiamond\log n}{\sqrt{\sigma_\blackdiamond\log n}}\right].\notag
\end{align*}
Using the central limit theorem together with Skorokhod's representation theorem, we can assume that, without loss of generality, the probability space is chosen in such a way that
$$
Z_n^\blackdiamond\ind_{\{c+o(1)\le Z_n^\blackdiamond\le (n-\sigma_\blackdiamond\log n)/\sqrt{\sigma_\blackdiamond\log n}\}}\toas N(0,1)\ind_{\{N(0,1)\ge c\}}.
$$
Additionally, the sequence on the left-hand side is uniformly integrable since the sequence $(\bE (Z_n^\blackdiamond)^2)_{n\ge 0}$ is bounded. This can be observed using that, by definition of $S_n^\blackdiamond$, it holds that
\begin{align*}
\bE S_n^\blackdiamond=\sum_{k=1}^n\bE\Big[ \text{Bern}\Big(\frac{\sigma_{\blackdiamond}}{k}\Big)\Big]=\sum_{k=1}^n\frac {\sigma_\blackdiamond}k=\sigma_\blackdiamond\Big(\log n +\gamma+O\Big(\frac 1n\Big)\Big),
\end{align*}
with $\gamma$ being the Euler-Mascheroni constant, and
\begin{align*}
\var S_n^\blackdiamond=\sum_{k=1}^n\var \Big[ \text{Bern}\Big(\frac{\sigma_{\blackdiamond}}{k}\Big)\Big]=\sum_{k=1}^n\frac {\sigma_\blackdiamond}k\Big(1-\frac{\sigma_\blackdiamond}{k}\Big)\tosim\sigma_\blackdiamond\log n.
\end{align*}
This implies the convergence of expectation, and thus, combined with~\eqref{eq:asym_1} we obtain
\begin{align}\label{eq:asym_2}
\bE\big[S_n^\blackdiamond\ind_{\{n-d\le Z_n^\blackdiamond\le n\}}\big]
&\tosim \sqrt{\sigma_\blackdiamond\log n}\cdot\bE[N(0,1)\ind_{\{N(0,1)\ge c\}}]+\sigma_\blackdiamond\log n \cdot\Phi(-c)\notag\\
&	\hspace*{+2.2mm}=\sqrt{\sigma_\blackdiamond\log n}\cdot\frac{e^{-c^2/2}}{\sqrt{2\pi}}+\sigma_\blackdiamond\log n \cdot\Phi(-c).
\end{align}
Inserting \eqref{eq:asym_1} and~\eqref{eq:asym_2} into~\eqref{eq:asym_Weyl_sum} yields
\begin{align*}
\sum_{\ell=0}^d(d-\ell)\cdot\blackdiamond(n,n-\ell)
&	\tosim n!\left(\sqrt{\sigma_\blackdiamond\log n}\cdot\frac{e^{-c^2/2}}{\sqrt{2\pi}}-c\sqrt{\sigma_\blackdiamond\log n}\cdot\Phi(-c)\right)\\
&	\tosim	n!\sqrt{\sigma_\blackdiamond\log n}\left(\frac{e^{-c^2/2}}{\sqrt{2\pi}}-c\cdot\Phi(-c)\right).
\end{align*}
Combining this with~\eqref{eq:stat_dim_Weyl_A} and~\eqref{eq:asym_D^A(N,n)_critical} leaves us with
\begin{align*}
\bE\Delta(W_{n,d}^\blackdiamond)\tosim \sqrt{\sigma_\blackdiamond\log n}\left(\frac{e^{-c^2/2}}{\sqrt{2\pi}\Phi(-c)}-c\right),
\end{align*}
which completes the proof.
\end{proof}

\bigskip

\section*{Acknowledgement}

Most of this work was carried out during the Trimester Program \textit{The Interplay between High Dimensional Geometry and Probability} at the Hausdorff Research Institute for Mathematics in Bonn.\\
The authors were supported by the DFG priority program SPP 2265 \textit{Random Geometric Systems}. TG and ZK also acknowledge support by the German Research Foundation (DFG) under Germany's Excellence Strategy  EXC 2044 -- 390685587, \textit{Mathematics M\"unster: Dynamics -- Geometry -- Structure}.

\vspace*{0.5cm}
\bibliography{bibliography}
\bibliographystyle{abbrv}

\end{document}